\documentclass{article}
    \usepackage[title,titletoc,toc]{appendix}
\usepackage{fullpage}
\usepackage{pgfplots}
\pgfplotsset{compat=1.8}
\usepackage{setspace}
\usepackage{epstopdf}
\usepackage[utf8]{inputenc}

\usepackage{comment}
\excludecomment{confidential}

\usepackage{dsfont}
    \usepackage{graphicx}
    \usepackage{amsmath, amssymb, latexsym, amsthm, url}
    \usepackage{mathrsfs,color,fancyhdr}
    \usepackage{cite}
    \usepackage{prettyref}

      \newtheorem{thm}{Theorem}[section]
      \newcommand{\bthm}{\begin{thm}} \newcommand{\ethm}{\end{thm}}
      \newtheorem{prop}[thm]{Proposition}
      \newcommand{\bprp}{\begin{prop}} \newcommand{\eprp}{\end{prop}}
      
      \newtheorem{fact}[thm]{Fact}
      \newcommand{\bfct}{\begin{fact}} \newcommand{\efct}{\end{fact}}
      \newtheorem{prob}[thm]{Problem}
      \newcommand{\bprb}{\begin{prob}} \newcommand{\eprb}{\end{prob}}
      \newtheorem{lem}[thm]{Lemma}
      \newcommand{\blem}{\begin{lem}} \newcommand{\elem}{\end{lem}}
      \newtheorem{claim}[thm]{Claim}
      \newcommand{\bclm}{\begin{claim}} \newcommand{\eclm}{\end{claim}}
      \newtheorem{cor}[thm]{Corollary}
      \newcommand{\bcor}{\begin{cor}} \newcommand{\ecor}{\end{cor}}
      \newtheorem{conj}[thm]{Conjecture}
      \newcommand{\bcnj}{\begin{conj}} \newcommand{\ecnj}{\end{conj}}
      \theoremstyle{definition}
      \newtheorem{defn}[thm]{Definition}
      \newcommand{\bdfn}{\begin{defn}} \newcommand{\edfn}{\end{defn}}
      \theoremstyle{remark}
      \newtheorem{rem}[thm]{Remark}
      \newcommand{\brem}{\begin{rem}} \newcommand{\erem}{\end{rem}}
      \newtheorem{cnv}[thm]{Convention}
      \newcommand{\bcnv}{\begin{cnv}} \newcommand{\ecnv}{\end{cnv}}
      \newtheorem{exam}[thm]{Example}
      \newcommand{\bexm}{\begin{exam}} \newcommand{\eexm}{\end{exam}}
      \newcommand{\bpf}{\begin{proof}} \newcommand{\epf}{\end{proof}}
      \newtheorem{exer}[thm]{Exercise}
      \newcommand{\bexer}{\begin{exer}} \newcommand{\eexer}{\end{exer}}
      \newcommand{\ben}{\begin{enumerate}}
      \newcommand{\een}{\end{enumerate}}
      \newcommand{\bit}{\begin{itemize}}
      
      \newcommand{\eit}{\end{itemize}}

        \newcommand{\reals}{{\mathbb R}}

        \newcommand{\Z}{{\mathbb Z}}
        
        \newcommand{\N}{{\mathbb N}}

        \newcommand{\pr}{\mathbb{P}}
        
        \newcommand{\Pset}{\mathbf{p}}
        \newcommand{\Qset}{\tilde{\mathbf{p}}}
        \newcommand{\Eset}{\mathbf{E}}
	\newcommand{\EQset}{\tilde{\textbf{E}}}
        
        \newcommand{\E}{\mathbb{E}}
        \newcommand{\ZQ}{\mathcal{Z}}

        \newcommand{\renP}{\mathbf P}
	\newcommand{\renE}{\mathbf E}
	\newcommand{\renZ}{\mathbf Z}        
        \newcommand{\ignore}[1]{}
       
       \newcommand{\one}{\mathds{1}}
	\newcommand{\dif}{\text{d}}

\usepackage{color}

\definecolor{Red}{rgb}{1,0,0}
\definecolor{Blue}{rgb}{0,0,1}
\definecolor{Olive}{rgb}{0.41,0.55,0.13}
\definecolor{Yarok}{rgb}{0,0.5,0}
\definecolor{Green}{rgb}{0,1,0}
\definecolor{MGreen}{rgb}{0,0.8,0}
\definecolor{DGreen}{rgb}{0,0.55,0}
\definecolor{Yellow}{rgb}{1,1,0}
\definecolor{Cyan}{rgb}{0,1,1}
\definecolor{Magenta}{rgb}{1,0,1}
\definecolor{Orange}{rgb}{1,.5,0}
\definecolor{Violet}{rgb}{.5,0,.5}
\definecolor{Purple}{rgb}{.75,0,.25}
\definecolor{Brown}{rgb}{.75,.5,.25}
\definecolor{Grey}{rgb}{.5,.5,.5}

    \date{\today}

    \title{Scaling limit of wetting models in 1+1 dimensions pinned to a shrinking strip}
    \author{Jean-Dominique Deuschel\footnote{Technische Universit\"at Berlin, {\tt deuschel@math.tu-berlin.de}}\, and Tal Orenshtein\footnote{Technische Universit\"at Berlin and Humboldt Universit\"at Berlin, {\tt orenshtein@tu-berlin.de}}
   }
 
\begin{document}

  \date\today
    \maketitle

    \begin{abstract}
We consider wetting models in 1+1 dimensions with a general pinning function on a shrinking strip.
 We show that under diffusive scaling, the interface converges in law to to the reflected Brownian motion, 
 whenever the strip size is $o(N^{-1/2})$ and the pinning function is close enough to critical value of the 
 so-called $\delta$-pinning model
of Deuschel, Giacomin, and Zambotti \cite{deuschel2005scaling}. 
As a corollary, the same result holds for the constant pinning strip wetting model at criticality 
with order $o(N^{-1/2})$ shrinking strip.
 \end{abstract}
\hfil



\tableofcontents

\section{Introduction}\label{sec:intro}

\subsection{The standard wetting model}\label{subsec:standard model}

Let $(S_k)_{k=0,1,...}$ be a random walk with increments $S_k-S_{k-1}$, $k\ge 1$, which are i.i.d with law $\pr$. 
We assume that $\pr$ has a continuous probability density of the form 
$\rho(x)=\frac{1}{\kappa}e^{-V(x)}$, so that $V$ is symmetric and strictly convex (in the sense that $V$ in $C^2$ and $V''(x)\in[1/c,c]$ for some $c>1$). 
Symmetry then implies that $\E [S_1]=0$. We assume also that the normalizing constant $\kappa$ is so that $\E [S_1^2]=1$.

Denote by by $\pr _x$ the law of $S$, starting at $x\in\reals$,
and let $\E _x$ be the corresponding expectation function. 
For ease of notation we let $\pr=\pr_0$ and $\E=\E_0$. 

As a convention 
throughout the paper
expressions of the form $\pr_x[A,S_N=y]=\E_x[\one_{A}\one_{\{y\}}(S_N)]$, are to be read as the density 
of $S_N$ at $y$ with respect to the measure $\pr_x$ on the event $A$,
More explicitly, for a random variable $Y$,
\begin{equation}\label{eq:density on events}
\E_x[Y\one_{\{y\}}(S_N)]:=\lim_{\epsilon\to 0}\frac{1}{2\epsilon}\E_x[Y\one_{[y-\epsilon , y+ \epsilon]}(S_N)]. 
\end{equation}

\emph{The standard wetting model}, also called the $\delta$-pinning model, was introduced in \cite{deuschel2005scaling}. 
It is a measure on $\reals^N_+$ where two possible boundary conditions are considered, 
\emph{free} and \emph{constraint}. 
The constraint case is defined by 
\begin{equation}\label{eq:StandardPinningConstrainedDef1}
P_{\beta,N}^c(\dif x) =\frac{1}{Z_{\beta,N}^c }
\exp\left(-\sum_{i=1}^{N}V(x_{i}-x_{i-1})\right)
\prod _{i=1}^N \left(\dif x_i\one_{[0,\infty)}+e^\beta \delta_0(\dif x_i)\right),
\end{equation}
where $x_0=x_N=0$. 
Analogously, the free case is defined by 
\begin{equation}\label{eq:StandardPinningFreeDef1}
P_{\beta,N}^f (\dif x) =\frac{1}{Z_{\beta,N}^f }
\exp\left(-\sum_{i=1}^{N}V(x_{i}-x_{i-1})\right)
\prod _{i=1}^N \left(\dif x_i\one_{[0,\infty)}+ e^\beta \delta_0(\dif x_i)\right),
\end{equation}
where $x_0=0$. 
Here $\dif x_i$ is the Lebesgue measure on $\reals$, and the partition functions $Z_{\epsilon,N}^c$ and $Z_{\epsilon,N}^f$ are normalizing constants so that 
$P_{\beta,N}^c$ and $P_{\beta,N}^f$ are probability measures on $\reals_+^N$. 

A remarkable localization transition was proved in \cite{deuschel2005scaling} using a 
renewal structure naturally corresponding to the model. On the heuristic level, the conditioned law on the contact set,  
the excursions from zeros are independent and their law is independent of the pinning parameter. 
Hence one expects to see that under the conditioning, the (appropriately rescaled interpolated) excursions converge the Brownian excursions. 
To analyze the full path one therefore needs an understanding of the contact set distribution.
Whenever $N$ is large, the contact set looks like a renewal process with inter-arrival distribution expressed in terms of the Green function of the walk.

In particular, making the above intuition accurate and quantitative, in \cite{deuschel2005scaling} 
(and tailored for renewal theory techniques in \cite{caravenna2006sharp}) the authors proved that 
there exists some $\beta_c\in\reals$, explicitly defined in \eqref{eq:def of critical beta} below, 
so that under the standard diffusive 
scaling and interpolation to continuous paths on $[0,1]$ the following a limit in distribution holds, with the following laws:
\begin{itemize}
 \item For $\beta<\beta_c$, the Brownian meander (free case) or the Brownian excursion (constrained case).
 \item For $\beta>\beta_c$, a mass-one measure on the constant zero function.
 \item For $\beta=\beta_c$, the reflecting Brownian motion (free case) or the reflecting Brownian bridge (constrained case).
 \end{itemize}
Moreover, $\beta_c$ is explicit in terms of the random walk density $\rho$. In particular, 
\begin{equation}\label{eq:def of critical beta}
e^{-\beta_c}=\sum_{n=1}^\infty f_n,  
\end{equation}
where $f_n:=\pr_0[\mathcal{C}_n, S_n=0]$ is the density of $S_n$ at zero on the event  
$\mathcal{C}_n=\{S_1\ge 0,...,S_n \ge 0\}$ (remember \eqref{eq:density on events}). 
We remark already at this stage that 
\begin{equation}\label{eq:asymptotic of fn} 
f_n=\frac{1}{\sqrt{2\pi}}n^{-3/2} + o(n^{-3/2})   
\end{equation}
 {and moreover, in the Gaussian case $V(x)=\frac{1}{2}x^2$,  the error term is identically zero
\cite[Lemma 1]{deuschel2005scaling} 
(see also \eqref{eq:def of f n a x y} and a few lines below it) and in particular $\beta_c=\log \left(\frac{1}{\sqrt{2\pi}}\sum_{n=1}^\infty n^{-3/2}\right)$.
}
\subsection{The strip wetting model with general pinning function}\label{subsec:the model}
 
\emph{The strip wetting model} is the analogous measures on $\reals_0^N$ which we now define.
Fix a one-parameter family of functions $\{\varphi_a, a\in(0,a_0]\}$, 
so that $\varphi_a:\reals_+\to\reals$ and $\int_0^a e^{\varphi_a(x)}{\dif}x$ is finite for $0<a\le a_0$, 
where ${\dif}x$ is the Lebesgue measure on $\reals$. 
Let $\mathcal{C}_N$ be the event $\{S_1\ge 0,...,S_N\ge 0\}$. 
We define now $\pr^\alpha_{ \varphi_a,N}$ and $\alpha\in\{c,f\}$. 
Whenever we would like to emphasize the pinning functions we also call them \emph{the $\varphi_a$-wetting model}.     
The case of free boundary conditions is defined by the Radon-Nikodym derivative
\begin{equation}\label{eq:ModelDefFree}
{\dif{\pr}^f_{ \varphi_a,N}} (S)= \frac{1}{\ZQ^f_{ \varphi_a,N}}
 \exp\left(\sum_{n=1}^N\varphi_a(S_n)\one_{[0,a]}(S_n)\right)\one_{\mathcal{C}_N}{\dif\pr}(S),
\end{equation}
while the constraint case is defined by the Radon-Nikodym derivative
\begin{equation}\label{eq:ModelDefConstraint}
{\dif{\pr}^c_{ \varphi_a,N}} (S)= \frac{1}{\ZQ^c_{ \varphi_a,N}}
 \exp\left(\sum_{n=1}^N\varphi_a(S_n)\one_{[0,a]}(S_n)\right)\one_{[0,a]}(S_{N})\one_{\mathcal{C}_N}{\dif\pr}(S).
\end{equation}

The normalizing constants $\ZQ^f_{ \varphi_a,N}$ and $\ZQ^c_{ \varphi_a,N}$ are called the partition functions.
When we want to specify the initial and ending points, we also define the density at $y\in \reals_+$ by
\begin{equation}\label{eq:stripModelGenFucn}
\ZQ^c_{ \varphi_a,N}(x,y)= \E_x\left[ \exp\left(\sum_{n=1}^N\varphi_a(S_n)\one_{[0,a]}(S_n)\right)\one_{\{y\}}(S_{N})\one_{\mathcal{C}_N} \right], x\in \reals_+,\,\,N\ge 1,
\end{equation}
so that
\begin{equation*}
\ZQ^c_{\varphi_a,N}= \int_{0}^a \ZQ^c_{ \varphi_a,N}(0,y)\dif y.
\end{equation*}

The connection between the strip and the standard wetting models is discussed in Appendix \ref{appendix:app3}.

\subsection{Main results}\label{sec:main results}

As mentioned in the Introduction this paper deals with strip models approximating the critical standard wetting model 
in a regularizing way. The regularization is due to the fact we allow the pinning functions $\varphi_a$ to be smooth.
The approximation is due to the fact the strip size $a$ is taken to zero with the model size $N$.

As we shall see in Chapter \ref{sec:application for critical strip wetting}, as an application we 
prove that the strip wetting model with constant pinning $\beta_c(a_N)$ 
has the same asymptotic behavior as the critical standard wetting model, whenever 
the strip size $a_N$ is decaying asymptotically faster than $\frac{1}{\sqrt N}$. 

We start with some notations. For a path $(S_i)_{i\ge 0}$, let $\tau_0^a=0$, 
$\tau_{j}^a=\inf\{n>j:S_j\in [0,a]\}$, $\ell_N^a=\sup\{k:S_k\in [0,a]\}$.
Let $\mathcal{A}^a_N=\{\frac{\tau_j}{N}: j\le \ell^a_N\}\subset[0,1]$ be the zero-set up to time $N$. 
Define now for $A=\{t_1<...,t_{|A|}\}$, $0=:t_1<...<t_{|A|}\le N$,
\begin{equation}\label{eq:def of measure on sets with varphi}
\Qset^\alpha_{ \varphi_a,N}(\mathcal{A}^a_N=A/N):=
\pr^\alpha_{ \varphi_a,N} (\tau^a_i=t_i, i\le \ell^a_N), 
\end{equation}
and $\EQset^\alpha_{ \varphi_a,N}$, $\alpha\in\{c,f\}$, the corresponding expectation.
In a somewhat abuse of notation we use $\Qset^c_{ \varphi_a,N}(A)$ and 
$\Qset^c_{ \varphi_a,N}(\mathcal{A}^a_N=A/N)$ with no distinction.
Note that by definition $\Qset^c_{ \varphi_a,N}(A)=0$ whenever $\ell^a_N(A)<N$.


\begin{defn}\label{def:cond A} 
 We say that that $(\varphi_a)_{0<a<a_0}$ satisfies Condition (A) if there is a constant $C>0$
 such that, uniformly in $x\in [0,a]$,
 $$-C\le \frac{1}{a}\log\int_0^a e^{\varphi_a(x)-\beta_c}\dif x\le C, $$ 
 for all $0<a<a_0$. Where $\beta_c$ was defined in \eqref{eq:def of critical beta}. 
\end{defn}

\begin{rem}
 Note that Condition (A) guarantees that that for $N$ fixed, the $\varphi_a$-wetting model 
 converges weakly to the critical standard wetting model as $a$ tends to $0$, see more in Appendix \ref{appendix:app3}.    
\end{rem}

The content of the next theorem is a scaling limit of the contact sets. 
For that we shall use the Matheron
topology on close real sets \cite{matheronrandom}. 
The basic notions can be found in \cite[page 209]{giacomin2007random},
\cite[Chapter 7]{deuschel2005scaling}, and \cite[Appendix B]{caravenna2006sharp}. 

\begin{defn}
Let $B$ be a standard one-dimensional Brownian motion (resp.\ bridge from $0$ to $1$).
We call the random set $\{t\in[0,1]: B_t=0\}$ the \emph{Brownian motion (resp.\ bridge) zero-set}.
\end{defn}

\begin{thm}\label{thm:main thm general zero set limit}
Fix some sequence $a_N=o(N^{-1/2})$. Assume that $\varphi_a$ satisfies Condition (A) from definition \ref{def:cond A}. 
Then under $\Qset^{\alpha}_{\varphi_{a_N},N}$, seen as a probability measure on the Matheron topological space of closed sets of $[0,1]$, 
the set $\mathcal{A}_N$ is converging in distribution to the Brownian motion zero-set for $\alpha = f$, and to the
Brownian bridge zero-set for $\alpha =c$.
\end{thm}

We also have a full path scaling limit.
\[
X^{(N)}_t:=\frac{1}{N^{1/2}} X_{\lfloor Nt \rfloor }+  \frac{1}{N^{1/2}} (Nt -\lfloor Nt \rfloor) (X_{\lfloor Nt \rfloor+1} - X_{\lfloor Nt \rfloor}). 
\]
\begin{thm}\label{thm:main thm general path limit}
If $a_N=o(N^{-1/2})$ then the process $(X^{(N)}_t)_{t\in[0,1]}$ under $\pr^{\alpha}_{\varphi_{a_N},N}$
converges weakly in $C[0,1]$ to the reflected Brownian motion on $[0, 1]$ for $\alpha = f$ and to the reflected
Brownian bridge on $[0, 1]$ for $\alpha =c$.
\end{thm}

\subsection{Examples}\label{subsec:examples}

\subsubsection{Constant pinning}\label{subsec:example: Strip wetting}
We call the model \emph{the strip wetting model with constant pinning} whenever the pinning function is constant 
on the strip, i.e., for some $\beta=\beta(a)\in\reals$
$\varphi_a(x)=\beta$, $x\in[0,a]$.

This model was suggested in Giacomin's monograph \cite[Equation (2.57)]{giacomin2007random} as an open problem,
and a major progress was done by Sohier \cite{sohier2013scaling,sohier2015scaling}. Application of our results 
in this case are presented in Section \ref{sec:application for critical strip wetting}. 

\subsubsection{Smooth approximation of the critical standard model}\label{subsec:example: Smooth approximation}
 {We construct a function $\varphi_a\in C^\infty(\reals)$ supported on $[0,a]$ so that it satisfies Condition (A) 
from Definition \ref{def:cond A}.

Let 
$$f(x):=\begin{cases}
           e^{-1/x} & x>0 \\
           0 &  x\le 0
           \end{cases}.$$ 
It is easy to verify that the derivatives of $f$ at $0$ vanish and hence it is $C^{\infty}(\reals)$.
Choose some $\epsilon(a)\to 0$ as $a\to 0$ with the rate of decay to be specified later-on and 
let 
\[
g_a(x)= \epsilon(a) + \frac{1}{a}
\frac{f(\frac{a-x}{\epsilon(a)})} {f(1-\frac{a-x}{\epsilon(a)}) + f(1\frac{a-x}{\epsilon(a)})}.
\]
It is easy to check that $\epsilon(a) \le g(x)\le 1/a + \epsilon(a)$, $g(x) = 1/a + \epsilon(a)$ if $x\le a-\epsilon{a}$, 
and $g(x) = \epsilon(a)$ if $x\ge a$. Therefore
$ (1/a + \epsilon(a))(a-\epsilon(a)) \le \int_0^a g(x)\dif x \le (1/a + \epsilon(a))a$. 
Therefore, choosing $\epsilon(a)\le a^2$ then there is some constant $C>0$ so that for all $a$ small enough 
$$e^{-Ca}\le 1 + a\epsilon(a) - \epsilon(a)/a + \epsilon(a)^2 \le \int_0^a g_a(x)\dif x \le 1 + a\epsilon(a)\le e^{Ca}.$$
We remark that $\exp(\beta_c)\equiv {\sqrt{2\pi}}/ \sum_{n\ge1} n^{-3/2}  \approx 0.961849$.
Set $\varphi_a(x):=\left(\beta_c + \log g_a(x)\right)\one_{\reals_+}(x), x\in\reals$, where $\epsilon(a)=a^2$.
See Figure \ref{fig:graph} for a graphical presentation.
Then $\varphi_a\in C^\infty([0,a])$ and satisfies Condition $A$ from Definition \ref{def:cond A}.  

\begin{figure}[!htb]
\centering
\includegraphics[scale=.9]{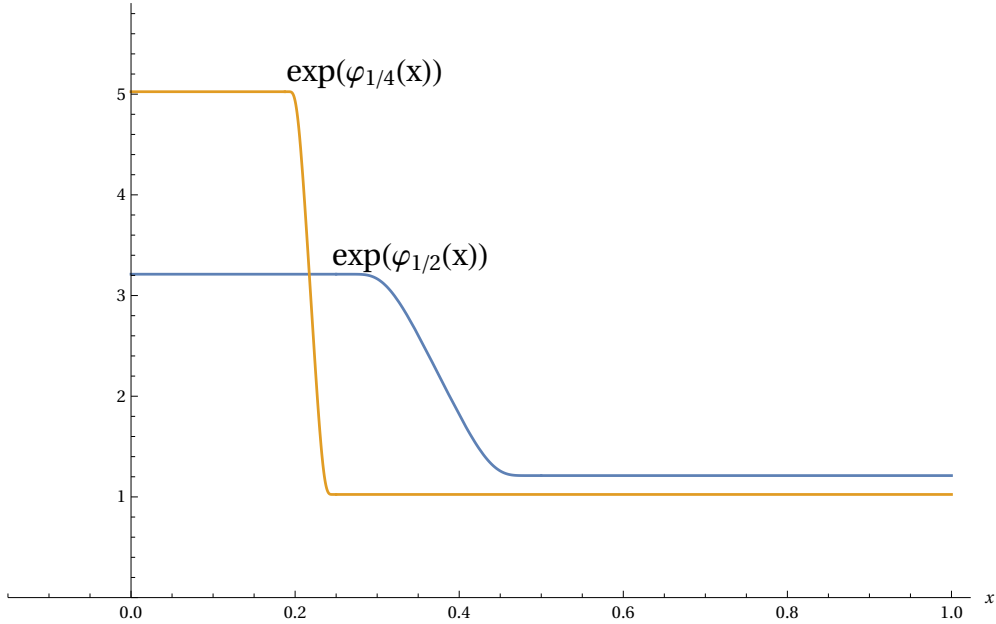}
\caption{The graph of $\exp(\varphi_a(x))$, $0\le x \le 1$, for $a=1/4$ and $a=1/2$.
}
\label{fig:graph}
\end{figure}

\subsection{Motivation: the dynamic of entropic repulsion with critical pinning}\label{subsec:Motivation}
 {
Take a family $\varphi_a\in C^2$ supported in $[0,a]$, $0<a<1$, which satisfies Condition (A) from Definition \ref{def:cond A}, see for example Chapter \ref{subsec:example: Smooth approximation}.  
Fix some $a>0$. We can easily construct a dynamic ${X}_t(x)$, $t\ge 0$, $x\in I_N:=\{0,1,...,N\}$, for which the measure  
$\pr^c_{\varphi_{a},N}$ defined in \eqref{eq:ModelDefConstraint} is a reversible equilibrium:
\[
  {X}_x(t) = -\int_{0}^{t} \partial_x H_{N}(X(s))\dif s +  \ell_t(x) + \sqrt{2} W_t (x) , \,\, x\in I_N, t\ge 0,
\]
with boundary conditions 
\[
 X_0(t)=X_N(t)=0, \,\, t\ge 0,
\]
initial law 
\[ 
(X_x(0))_{x\in I_N} \sim \pr^c_{\varphi_a,N}, 
\]
so that the local time process $\ell_t$ satisfies
\[
\dif \ell_x(t) \ge 0,\,\,  t\ge0, x\in I_N,  
\]
and
\[
 \int_0^\infty {X}_x(t)\dif \ell_x(t)=0,\,\,  x\in I_N, t\ge 0,
\]
$W(x)$, $x\in I_N$, are independent standard Wiener measures,
\[
\partial_x H_N(X):= \frac{\partial}{\partial X_x} H_N(X), 
\]
and the Hamiltonian
\[
H_N(X):=\sum_{x = 0}^N \varphi_{a}(x) + \frac{1}{2} \sum _{x=1}^{N} (X_x-X_{x-1})^2 + \frac{1}{2} X_0^2 + \frac{1}{2}X_N(N)^2.
\]
Let $X^{N}(t)$ be the diffusively rescaled and linearly interpolated path given by  
\[
X^{N}_y(t)=\frac{1}{N^{1/2}} X_{\lfloor Ny \rfloor }(t)+  \frac{1}{N^{1/2}} (Ny -\lfloor Ny \rfloor) (X_{\lfloor Ny \rfloor+1 }(t) - X_{\lfloor Ny \rfloor}(t)) ,\, t\ge 0,\, y\in [0,1]. 
\]
Our Theorem \ref{thm:main thm general path limit} states that if  $a=a_N=o(N^{-1/2})$, then
\[
(X_y^N(0))_{y\in[0,1]} \Rightarrow (\beta_y)_{y\in[0,1]}\,\,\,\,\ (*) 
\]
where $(\beta_y)_{y\in[0,1]}$ is the reflected Brownian bridge.

We expect that $\{X^N_y(tN^2),y\in[0,1], t\ge 0 \}$ is tight in $N\in\N$ and converges to a SPDE 
which is the natural reversible dynamic associated with $(\beta_y)_{y\in[0,1]}$.
The construction of the dynamic in finite volume for \emph{singular} drift was addressed in Funaki's lecture notes \cite[Chapter 15.2]{funaki2005stochastic} and 
in \cite{fattler2016construction} using Dirichlet form techniques. Due to our approach the construction becomes  
easy since it allows a \emph{smooth} drift so that $(*)$ still holds.

}

\subsection{Applications to strip wetting with constant pinning at criticality}\label{sec:application for critical strip wetting}

Sohier \cite{sohier2015scaling} considered the strip wetting model with constant pinning and proved that there is some $\beta_c(a)\in\reals$ 
so that off-criticality, the same path scaling limit results as in the standard wetting model hold true. 
Namely, in this case the limiting object is
\begin{itemize}
 \item Brownian meander (free case) or the Brownian excursion (constrained case), whenever $\beta<\beta_c(a)$, and 
 \item a mass-one measure on the constant zero function, whenever $\beta>\beta_c(a)$. 
 \end{itemize}
In particular, he proved also a corresponding statement on the off-critical contact set scaling limits.
Moreover, $\beta_c(a)$ is represented in terms of an eigenvalue of a natural Hilbert-Schmidt 
integral operator, see \cite{sohier2015scaling}, and Section \ref{sec:Detour - the associated integral operator and the critical value derivation}.  
 
The next theorem deals with critical value $\beta_c(a)$ of the constant pinning model for $a$ small. 
It states that the critical value $\beta_c$ of the standard wetting model
is well-approximated by $\beta_c(a)$.   
\begin{thm}\label{thm:critical wetting close to standard critical wetting}
There is are constants $C,D>0$ so that 
$${Da^2}\le \log a + \beta_c(a) -\beta_c\le C a $$ 
for all $a>0$ small enough.
In particular, the constant function $\varphi_a=\beta_c(a)$ satisfies Condition (A) from Definition 
\ref{def:cond A}, and moreover $ae^{\beta_c(a)}\to e^{\beta_c}$ as $a\to 0$.
\end{thm}
In particular, we have an analogous contact set and full path scaling limits in the critical case on shrinking strips:  
\begin{cor}\label{cor:crit wetting enjoys general thms}
Theorems \ref{thm:main thm general zero set limit} and \ref{thm:main thm general path limit}
hold true also for the critical constant pinning models, i.e.\ whenever $\varphi_a(x)=\beta_c(a), x\in[0,a]$.  
\end{cor}

 {
\begin{rem}
In \cite{sohier2013scaling} the critical contact set with free boundary conditions was considered, 
for fixed size $a$ of the strip.
That paper states that the rescaled contact set converges to a random set which distribution is absolutely 
continuous but not equal to the Brownian motion zero-set.
In our case when $a=a_N=o(N^{-1/2})$ the limit is the Brownian motion zero-set.
Also we prove the full path convergence to reflected Brownian motion.
Although [Soh13] does not contradict our results, since it deals with fixed $a$,
we believe that there is a gap in the proof of Theorem 1.5, and in particular in Lemma 3.3.
Also, the case $a=a_N=N^{-\gamma}$
with $\gamma<1/2$ remains open, see 
Section \ref{sec:remarks} for the case $\gamma=1/2$.
\end{rem}

}

\section{Comparing excursion kernels}\label{sec:comparing kernels}

Define the excursion kernel density
\begin{equation}\label{eq:def of f n a x y}
f^a_n(x,y)\dif y := \pr _x [S_1>a,...,S_{n-1}>a,S_n \in \dif y] 
\end{equation}
for $n\ge 2$, where $f_1^a(x,y)\dif y =  \pr _x [S_1 \in \dif y]$.
Let 
$$f^a_n:=f^a_n(0,0),$$ and we omit the up-case $a$ whenever $a=0$, that is 
$$f_n=f^0_n.$$

The first observation is that the $f^a_n$ approximate the corresponding $f_n$.
\begin{lem}\label{lem:comparing excursions kernel to pinning}
The following hold:
\begin{itemize}
 \item $f^a_n$ is symmetric: $f^a_n(x,y)=f^a_n(y,x)$ for all $x,y\in[0,a]$, $n\ge 1$. 
 \item $f^a_n(x,y)$ is monotonously increasing in $x,y\in[0,a]$.
 \item $f^a_n(a,a)=f_n$.
  \end{itemize}
In particular,
\begin{equation}\label{eq: upper bound fan fn} 
\frac{f^a_n}{f_n} \le \frac{f^a_n(x,y)}{f_n} \le 1 
\end{equation}
for all $x,y\in[0,a]$ and $n\ge 1$.
Moreover, $\frac{f^a_n}{f_n}$ decreases in $a$ and tends to 1 as $a\to 0$, for all $n$.
\end{lem}
\begin{proof}
For the first two properties, one uses the corresponding assumptions on $\rho$ on the following explicit expression for the 
densities  
\[
f^a_{n+1}(x,y)=\int_a^{\infty}...\int_a^{\infty}\rho(s_1-x)\rho(s_2-s_1)\cdots\rho(s_{n}-s_{n-1})\rho(y-s_{n})\dif s_1\cdots\dif s_{n}. 
\]
The last property follows, e.g., by the change of variables $s_i\to s_i + a, i=1,...,n$.   
\end{proof}

Let 
\begin{equation}\label{eq:def of P(n)}
P^a_x(n):=\pr_x[S_1> a,...,S_n> a], \text{ and  } P(n):=P^0_0(n).  
\end{equation}
Note that $P^a_x(n)$ is (continuously) increasing in $x\in[0,a]$.
In particular, $P^a_0(n)\le P^a_x(n) \le P^a_a(n)=P(n)$ for $x\in[0,a]$.  
For the right part a classical result is
\[
P^0_0(n)\sim \frac{1}{\sqrt{2\pi}} n^{-1/2}.
\]
The following is a weak version of Sohier \cite[Lemma 2.2.]{sohier2015scaling}. 
\begin{lem}
 There is a monotonously decreasing function $C^a(x):[0,a]\to\reals+$ so that $C^a(a)=1$, $C^a(0)>0$ and 
\[
P^a_x(n)\sim \frac{C^a(x)}{\sqrt{2\pi}} n^{-1/2}.
\]
\end{lem}
\begin{proof}
If we set $C^a(x):= \mathbf{P}[H_1\ge a-x]$, the asymptotic equivalence in the line above is the content of \cite[Lemma 2.2.]{sohier2015scaling}, where 
$H_1$ is the so called first ascending ladder point. 
The proof is done by noticing that $H_1$ is defined to be a non-negative random variable.
\end{proof}
 Putting the last statements together, we get that there is a monotonously decreasing function $C^a(x):[0,a]\to\reals+$ so that $C^a(a)=1$ $C^a(0)>0$ and 
 \begin{equation}\label{eq:comparing Pan in terms of Pn}
C^a(0) \sim \sqrt{2\pi}n^{1/2}P^a_0(n)\le \sqrt{2\pi}n^{1/2}P^a_x(n) \le \sqrt{2\pi}n^{1/2}P(n)  \sim 1
 \end{equation}
 for $x\in[0,a]$.  

 As a corollary we have
 \begin{cor}
 Assume that $a=a_n\to 0$. Then uniformly in $x_n\in [0,a_n]$  
  \[
 \sqrt{2\pi}n^{1/2}P^{a_n}_{x_n}(n) \to 1 \text{ as } n\to \infty,
 \]
 or equivalently $P^{a_n}_{x_n}(\cdot)\sim P(\cdot)$. 
 \end{cor}
 \begin{proof}
Indeed,
  \[
 1=\liminf_{n\to\infty} C^{a_n}(0)\le \liminf_{n\to\infty} \sqrt{2\pi}n^{1/2}P^{a_n}_{x_n}(n) \le \limsup_{n\to\infty}\sqrt{2\pi}n^{1/2}P(n)= 1
 \]
 \end{proof}

\subsubsection*{Approximating $f_n^a$ in terms of $f_n$}\label{sec:estimating f}
 The main goal of this section is to estimate $f_n^a$ in terms of $f_n$ and $a$. The next lemma actually supplies
 upper and lower bounds, but for the results of the paper we shall only use the lower bound.

\begin{lem}\label{lem:estimating fn(a)}
There are constants $0<c_0,\tilde{c}_0,c_1,\tilde{c}_1$ so that for all $0\le a \le 1$ and $n\ge 1$
\[
 \exp (- c_0 a -\tilde{c}_0 a^2 ) \le f^a_n/f_n \le \exp(-c_1 a +\tilde{c}_1 a^2 ).
\]
In particular, there is some $0<a_0$ and constants $C_0,C_1$ so that for all $0\le a\le a_0$ and $n\ge 1$
\begin{equation}\label{eq: lower bound comparison fan fn} 
 \exp(-C_0 a)  \le f^a_n/f_n \le \exp(-C_1 a).
\end{equation}
\end{lem}
\begin{proof}
Denote by $A_n(y)$ the event $\{S_1>0,...,S_{n-1}>0,S_n=y\}$, where we remind the reader that 
by convention we write $\pr_x[A_n(y)]$ for the density of $S_n$ at $y$ with respect to $\pr_x$ on the event $\{S_1>0,...,S_{n-1}>0\}$.
In other words, $\pr_x[A_n(y)]=f^0_n(x,y)$.
We first note that $f^a_n=f^a_n(0,0)=f^0_n(-a,-a)=\pr_{-a}[A_n(-a)]$, by stationarity.
Taking a derivative from the right-most expression we get 
\[
\frac{\partial}{\partial a}f^a_n = - \E_{-a}[V'(S_1+a)\one_{A_n(-a)}] - \E_{-a}[V'(S_{n-1}+a)\one_{A_n(-a)}].
\]
On the event ${A_n(-a)}$ the random variables $S_1+a$ and $S_{n-1}+a$ have the same distribution under $\pr_{-a}$ and therefore
\[
\frac{\partial}{\partial a}f^a_n = - 2\E_{-a}[V'(S_1+a)\one_{A_n(-a)}].
\]
In particular,
\[
\frac{\partial}{\partial a}{f^a_n}_{|a=0} = - 2\E_{0}[V'(S_1)\one_{A_n(0)}].
\]
A direct calculation for the second derivative yields
\begin{eqnarray*}
\frac{\partial^2}{\partial a^2}f^a_n &=& - 2\frac{\partial}{\partial a}\E_{-a}[V'(S_1+a)\one_{A_n(-a)}]\\
&=& 2\E_{-a}[( V'(S_1+a)^2-V''(S_1+a)+V'(S_1+a)V'(S_{n-1}+a))\one_{A_n(-a)}].
\end{eqnarray*}
A second order Taylor expansion reads
\[
\frac{f^a_n}{ f_n} - 1 = 
- \frac{2 a \E_{0}[V'(S_1)\one_{A_n(0)}]}{\pr_{0}[A_n(0)]} + 
 \frac{2a^2\E_{-a'}[( V'(S_1+a')^2 -V''(S_1+a')+V'(S_1+a')V'(S_{n-1}+a'))\one_{A_n(-a')}]}{\pr_{0}[A_n(0)]},
\]
where $0<a'<a$ (allowed to depend on $n$).
Therefore, the proof is finished once we show that both 
\begin{equation} \label{eq:first derivative}
 c_0 \le \E_{0}[V'(S_1)\one_{A_n(0)}]/\pr_{0}[{A_n(0)}]\le c_1
\end{equation}
and
\begin{equation} \label{eq:second derivative}
 -\tilde{c}_0 \le
\E_{-a'}[( V'(S_1+a')^2-V''(S_1+a')+V'(S_1+a')V'(S_{n-1}+a'))\one_{A_n(-a')}]/\pr_{0}[{A_n(0)}]
\le \tilde{c}_1
\end{equation}
hold for all $0\le a'\le a$ and $n\ge 1$.

To prove \eqref{eq:second derivative} it is enough to show that
\begin{equation} \label{eq:second derivative one}
\E_{-a'}[( V'(S_1+a')^2+V'(S_1+a')V'(S_{n-1}+a'))\one_{A_n(-a')}]/\pr_{0}[{A_n(0)}]
\le \tilde{c}_1
\end{equation}
and
\begin{equation} \label{eq:second derivative two}
\pr_{-a'}[A_n(-a')]/\pr_{0}[{A_n(0)}]=f^{(a')}(n)/f_n\le \tilde{c}_0.
\end{equation}

Let us first show \eqref{eq:first derivative}.
By reversibility of the walk (due to symmetry of $V$) 
$\pr_x[A_n(y)]=\pr_y[A_n(x)]$ for all $x,y\ge 0$. In particular,
\begin{equation*}
\pr_0[S_1>0,..,S_{n-1}>0,S_n\in[k,k+1]] =
\int_k^{k+1}\pr_{0}[S_1>0,..,S_{n-1}>0,S_n=x]\dif x= \int_k^{k+1}\pr_{x}[A_n(0)]\dif x
\end{equation*}
The Ballot theorem \cite[Theorem 1]{addario2008ballot} (and the form we shall use \cite[Theorem 5]{zeitouni2012branching}) therefore reads, for $k\le\sqrt{n}$
\begin{equation} \label{eq:Ballot interpretation}
c_2\frac{k+1}{n^{3/2}} \le \int_k^{k+1}\pr_{x}[A_n(0)]\dif x \le c_3\frac{k+1}{n^{3/2}},
\end{equation}
where the upper bound holds for all $k$.

For the upper bound we get from the right inequality of \eqref{eq:Ballot interpretation} and the 
three assumptions on $V$ that
\begin{eqnarray*}
\E_{0}[V'(S_1)\one_{A_n(0)}] &=&
\int_0^\infty V'(x)\rho(x) \pr_x[A_{n-1}(0)] \dif x \\
&=&
\sum_{k=0}^\infty \int_k^{k+1} V'(x)\rho(x) \pr_x[A_{n-1}(0)]\dif x\\
&\le&
\sum_{k=0}^\infty V'(k+1)\rho(k) \int_k^{k+1} \pr_x[A_{n-1}(0)]\dif x\\
&\le&
\frac{2c_3}{n^{3/2}} \sum_{k=0}^\infty V'(k+1) (k+1) \rho(k) =: \frac{c_1}{\sqrt{2\pi}n^{3/2}}.
\end{eqnarray*}

For the lower bound we get from the left inequality of \eqref{eq:Ballot interpretation} that
\begin{eqnarray*}
\E_{0}[S_1\one_{A_n(0)}] &\ge&
\int_1^2 V'(x)\rho(x) \pr_x[A_{n-1}(0)] \dif x \\
 &\ge&
 \max_{[1,2]}\{V'(x)\rho(x)\}\frac{2c_3}{n^{3/2}} =:\frac{c_0}{\sqrt{2\pi}n^{3/2}}.
\end{eqnarray*}
Using \eqref{eq:asymptotic of fn} and the fact that $\pr_{0}[{A_n(0)}]=f_n$, \eqref{eq:first derivative} 
is now proved.

We now prove \eqref{eq:second derivative one}. We have
\begin{eqnarray*}
\E_{-a}[( V'(S_1+a)^2+V'(S_1+a)V'(S_{n-1}+a))\one_{A_n(-a)}] &=&
\E_{0}[V'(S_1)^2+V'(S_1)V'(S_{n-1}))\one_{S_1>a,...,S_{n-1}>a,S_0=0}]\\
&\le&
\E_{0}[(V'(S_1)^2+V'(S_1)V'(S_{n-1}) )\one_{S_1>0,...,S_{n-1}>0,S_0=0}]
\end{eqnarray*}
by writing the terms in the explicit integral form. 
Now, as in the proof of \eqref{eq:first derivative}
\begin{eqnarray*}
\E_{0}[V'(S_1)^2\one_{S_1>0,...,S_{n-1}>0,S_0=0}] &\le&
\sum_{k=0}^\infty V'(k+1)^2 \rho(k)  \int_k^{k+1} \pr_x[A_{n-1}(0)]\dif x\\
&\le&
\frac{2c_3}{n^{3/2}} \sum_{k=0}^\infty (k+1)V'(k+1)^2 \rho(k)=:\frac{c_5}{\sqrt{2\pi}n^{3/2}}.
\end{eqnarray*}
For the term $\E_{0}[V'(S_1)V'(S_{n-1}) \one_{S_1>0,...,S_{n-1}>0,S_0=0}]$, note that
\begin{eqnarray*}
\pr_0[S_1+y>0,..,S_{n-1}+y>0,S_n+y \in[k,k+1]] &=&
\pr_y[S_1>0,..,S_{n-1}>0,S_n\in[k,k+1]] \\&=&
\int_k^{k+1}\pr_{y}[S_1>0,..,S_{n-1}>0,S_n=x]\dif x\\
&=& \int_k^{k+1}\pr_{x}[A_n(y)]\dif x.
\end{eqnarray*}
We shall use a general variation of The Ballot Theorem:
for $0\le y\le k+1 \le \sqrt{n}/2$,
\begin{equation} \label{eq:Ballot's corollary interpretation}
\int_k^{k+1}\pr_{x}[A_n(y)]\dif x \le c_5\frac{(k+1)(y+1)^2}{n^{3/2}}
\end{equation}
(see \cite[Corollary 2]{zeitouni2012branching}).
Now, as in the proof of \eqref{eq:first derivative}, by the symmetric roles of 
$x$ and $y$ in the integrand we have
\begin{eqnarray*}
\E_{0}[V'(S_1)V'(S_{n-1})\one_{A_n(0)}] &=&
 \int_0^\infty \int_0^\infty V'(x)\rho(x)V'(y)\rho(y) \pr_x[A_{n-2}(y)] \dif x \dif y\\
&=&
 2\int_0^\infty \int_{\sqrt{n}/2}^\infty V'(x)\rho(x)V'(y)\rho(y) \pr_x[A_{n-2}(y)] \dif x \dif y \\
 &+&
  \int_0^{\sqrt{n}/2} \int_0^{\sqrt{n}/2} V'(x)\rho(x)V'(y)\rho(y)  \pr_x[A_{n-2}(y)] \dif x \dif y\\
&=& (I) + (II)
\end{eqnarray*}
To prove $(I)$ we note first that by the local limit theorem $\pr_x[A_{n-2}(y)]\le 
\pr_x[S_{n-2}=y]\le C/\sqrt{n}$
for some constant $C$, uniformly on $x,y\in \reals$ and $n\ge 1$. 
In particular $\pr_x[S_{n-2}=y]$ is uniformly bounded from above by $C$.
Therefore, 
\begin{eqnarray*}
(I) &\le&
 2C\int_0^\infty \int_{\sqrt{n}/2}^\infty V'(x)\rho(x)V'(y)\rho(y) \dif x \dif y \\
&\le& \int_0^\infty V'(y)\rho(y) \left( \frac{2C}{\kappa} \int_{\sqrt{n}/2}^\infty V'(x)e^{-V(x)}\dif x \right) \dif y \\
&=&  \frac{2C}{\kappa}e^{-V(\sqrt{n}/2)}
\int_0^\infty V'(y)\rho(y) \dif y \\
&=&  \frac{2C}{\kappa}e^{-V(\sqrt{n}/2)}
\frac{1}{\kappa}e^{-V(0)} \\
&=&  \frac{2C}{\kappa^2}e^{-V(\sqrt{n}/2)}\\
&=&  o(n^{-3/2}),
\end{eqnarray*}
here we used the symmetry of $V$
to get $\int_0^\infty V'(y)\rho(y) \dif y = e^{-V(0)}=1$
and 
we used the strict convexity of $V$ to conclude that $e^{-V(\sqrt{n}/2)}$ is decaying faster than any polynomial.  
To see $(II)$, we first have that
\[
(II)\le
\sum_{k,l=0}^{\lfloor\sqrt{n}/2\rfloor} \int_l^{l+1} \int_{k}^{k+1} 
V'(x)\rho(x)V'(y)\rho(y) \pr_x[A_{n-2}(y)] \dif x \dif y
\]
By symmetry of $\pr_x[A_{n-2}(y)]$ the right hand side equals
\[
2\sum_{k=0}^{\lfloor\sqrt{n}/2\rfloor}\sum_{l=0}^{k} \int_l^{l+1} \int_{k}^{k+1} 
V'(x)\rho(x)V'(y)\rho(y) \pr_x[A_{n-2}(y)] \dif x \dif y,
\]
which is not larger than
\[
2\sum_{k=0}^{\lfloor\sqrt{n}/2\rfloor}\sum_{l=0}^{k}V'(k+1)\rho(k)
V'(l+1)\rho(l) \int_{k}^{k+1}  \int_l^{l+1} \pr_x[A_{n-2}(y)] \dif x \dif y.
\]
Using \eqref{eq:Ballot's corollary interpretation}, if $l\le k$ then
\[
 \int_l^{l+1}\int_k^{k+1} \pr_x[A_{n-2}(y)]\dif x\dif y \le
 \int_l^{l+1}c_5\frac{(k+1)(y+1)^2}{n^{3/2}} \dif y \le
c_5\frac{(k+1)(l+2)^2}{n^{3/2}}.
\]

\begin{eqnarray*}
(II) &\le&
2\sum_{k=0}^{\lfloor\sqrt{n}/2\rfloor}\sum_{l=0}^{k}
V'(k+1)\rho(k)V'(l+1)\rho(l) \int_{k}^{k+1}  \int_l^{l+1} \pr_x[A_{n-2}(y)] \dif x \dif y\\
&\le&
\frac{c_5}{n^{3/2}}\sum_{k=0}^{\sqrt{n}/2} \sum_{l=0}^k (l+2)^2V'(l+1)\rho(l)(k+1) V'(k+1)\rho(k)\\
&\le&
\frac{c_5}{n^{3/2}}\sum_{k=0}^{\infty}(k+1)^2 V'(k+1)^2\rho(k)\\
&=:&
\frac{c^6}{n^{3/2}}.
\end{eqnarray*}
\end{proof}

\subsection{Comparing the partition functions}\label{sec:comparing partition function}

\begin{lem}\label{lem:bounding gen fcn of varphi wetting in terms of standard}
Fix $\varphi_a$ and assume Condition (A) from Definition \ref{def:cond A} with the constant $C$.
Then, for there is a constant $C'$ and a positive decreasing function $C'(a)$ so that 
$C'(a)\to 1$ as $a\to 0$, and and for all $N\ge1$ we have
\begin{equation}\label{eq:constrained part fcns bds varphi in terms of near crit}
Z^c_{\beta_c-C'a,N} \le \ZQ^c_{ \varphi_a,N} \le Z^{c}_{\beta_c+C'a,N},
\end{equation}
and
\begin{equation}\label{eq:free part fcns bds varphi in terms of near crit}
C'(a)Z^{f}_{\beta_c-C'a,N} \le \ZQ^{f}_{ \varphi_a,N} \le Z^{f}_{\beta_c+C'a,N}.
\end{equation}
\end{lem}

\begin{proof}
We start with the constraint case.
\begin{eqnarray*}
\ZQ^c_{ \varphi_a,N} (0,y) &=&
\sum_{k=0}^{N-1}  \sum_{0=t_0<t_1<...<t_k<N} \int_0^a\cdots\int_0^a
\prod_{i=1}^k f^a_{t_i-t_{i-1}}(y_{i-1},y_{i})e^{\varphi_a(y_i)}f^a_{N-t_k}(y_{k},y)e^{\varphi_a(y)} \dif y_i  =: (*).
  \end{eqnarray*}
Using \eqref{eq: upper bound fan fn} and Condition (A) we have the following upper bounds.
 \begin{eqnarray*}
 (*) &\le&
 \sum_{k=0}^{N-1} \sum_{0=t_0<t_1<...<t_k<N} \int_0^a\cdots\int_0^a
f_{N-t_k} \prod_{i=1}^k f_{t_i-t_{i-1}}  e^{\varphi_a(y_i)} e^{\varphi_a(y)} \dif y_i \\
&=&
  e^{\varphi_a(y)} \sum_{k=0}^{N-1} (\int_0^a e^{\varphi_a(z)} \dif z )^k \sum_{0=t_0<t_1<...<t_k<N}
f_{N-t_k} \prod_{i=1}^k f_{t_i-t_{i-1}}  \\
&\le&
  e^{\varphi_a(y)} \sum_{k=0}^{N-1} e^{(\beta_c+C a)k} f_{N-t_k} \sum_{0=t_0<t_1<...<t_{k}<N}
\prod_{i=1}^k f_{t_i-t_{i-1}} \\
\end{eqnarray*}
Hence 
\begin{eqnarray*}
\ZQ^c_{ \varphi_a,N}=\int_0^a \ZQ^c_{ \varphi_a,N} (0,y)\dif y &\le&
\int_0^a  e^{\varphi_a(y)} \dif y \sum_{k=0}^{N-1} e^{(\beta_c+C a)k}  \sum_{0=t_0<t_1<...<t_{k}<N} f_{N-t_k} \prod_{i=1}^k f_{t_i-t_{i-1}}  \\
&\le & \sum_{k=0}^{N-1} e^{(\beta_c+C a)(k+1)}  \sum_{0=t_0<t_1<...<t_{k}<N}f_{N-t_k} \prod_{i=1}^k f_{t_i-t_{i-1}}  \\
&=& \sum_{k=1}^{N} e^{(\beta_c+C a)k}  \sum_{0=t_0<t_1<...<t_{k}=N}\prod_{i=1}^k f_{t_i-t_{i-1}}  \\
&=& Z^c_{\beta_c + C a ,N}. 
\end{eqnarray*}

Similarly for the lower bound, using \eqref{eq: lower bound comparison fan fn} and Condition (A), we get
\begin{eqnarray*}
(*) &\ge&
e^{\varphi_a(y)} \sum_{k=0}^{N-1} e^{(\beta_c-C a - C_0a )k} f_{N-t_k} \sum_{0=t_0<t_1<...<t_{k}<N} 
\prod_{i=1}^k f_{t_i-t_{i-1}} 
\end{eqnarray*}

Hence,
\begin{eqnarray*}
\ZQ^c_{ \varphi_a,N}=\int_0^a \ZQ^c_{ \varphi_a,N} (0,y)\dif y &\ge&
\int_0^a e^{\varphi_a(y)} \dif y \sum_{k=0}^{N-1} e^{(\beta_c-C a - C_0a )k} \sum_{0=t_0<t_1<...<t_{k}<N}f_{N-t_k}
\prod_{i=1}^k f_{t_i-t_{i-1}} \\
&\ge&
\dif y \sum_{k=1}^{N} e^{(\beta_c-C a - C_0a )k} \sum_{0=t_0<t_1<...<t_{k}=N} \prod_{i=1}^k f_{t_i-t_{i-1}} \\
&=& Z^c_{\beta_c - (C+C_0) a ,N}.
\end{eqnarray*}
Since $\ZQ^c_{ \varphi_a,N} =\int_{0}^a \ZQ^c_{ \varphi_a,N} (0,y)\dif y$, 
setting $C'=C+C_0$ we conclude the two bounds.

The free case is done in a similar manner. Indeed summing over the last contact before time $N$, we have
\begin{eqnarray*}
\ZQ^f_{ \varphi_a,N} &=&
\sum_{k=0}^N \int_0^a \ZQ^c_{a,\varphi_a,k}(0,y)  P^a_y(N-k) \dif y =: (*).
  \end{eqnarray*}
Using \eqref{eq:comparing Pan in terms of Pn}, the line before it, Condition (A), and the constraint case we have the following upper bound.
 \begin{eqnarray*}
 (*) &\le&
 \sum_{k=0}^N P(N-k) \int_0^a \ZQ^c_{a,\varphi_a,k}(0,y) \dif y \\
&\le&
\sum_{k=0}^N P(N-k) Z^c_{\beta_c + C a ,k}\\
 &=&
Z^f_{\beta_c + C a ,N}.
\end{eqnarray*}
Similarly for the lower bound, using Using \eqref{eq:comparing Pan in terms of Pn}, the line before it, \eqref{eq: lower bound comparison fan fn} and Condition (A), we get
 \begin{eqnarray*}
 (*) &\ge&
C^a(0)e^{-Ca}Z^f_{\beta_c - (C+C_0) a ,N}.
\end{eqnarray*}
Setting $C'(a):=C^a(0)e^{-Ca}$, we are done.
\end{proof}

\subsection{Derivative of $\varphi_a$-strip wetting with respect to near-critical standard wetting }\label{sec:Derivating varphi-strip wetting to near critical standard wetting}
In this section we will prove that for the contact set distribution, the $\varphi_a$-strip wetting is approximating
the corresponding a near-critical standard wetting model, that is so that 
it has critical pinning strength which is linearly perturbed by a constant multiple of the strip-size. 

Remember the definition in \eqref{eq:def of measure on sets with varphi} with the notations above it. 
We introduce the analog for the standard wetting model. 
\begin{equation}\label{eq:def of measure on sets for standard wettin}
\Pset^\alpha_{\beta,N}(\mathcal{A}_N=A/N):=
\pr^\alpha_{e^{\beta},N} (\tau_i=t_i, i\le \ell_N), 
\end{equation}
and $\Eset^\alpha_{\beta,N}$, $\alpha\in\{c,f\}$, the corresponding expectation.
Here as well, in a somewhat abuse of notation we use $\Pset^c_{\beta,N}(A)$ and 
$\Pset^c_{\beta,N}(\mathcal{A}_N=A/N)$ with no distinction.
Note again that by definition $\Pset^c_{\beta,N}(A)=0$ whenever $\ell_N(A)<N$.

\begin{lem}\label{lem:derivative varphi wetting to near crit standard}
Assume $\varphi_a$ satisfies Condition (A) from Definition \ref{def:cond A} with the constant $C$. 
Remember the definitions from \eqref{eq:def of measure on sets with varphi}.
There are some constants $c_i, i=1,...,6$, so that for $\alpha\in\{c,f\}$ 
\[
\frac{\dif\Qset^\alpha_{ \varphi_a,N}}{\dif\Pset^\alpha_{\beta_c+c_3a,N}} \le \frac{Z^\alpha_{\beta_c+c_1 a,N}}{C^\alpha(a)Z^\alpha_{\beta_c-c_2a,N}}
\]
and
\[
\frac{\dif\Qset^\alpha_{ \varphi_a,N}}{\dif\Pset^\alpha_{\beta_c-c_6a,N}} \ge \frac{Z^\alpha_{\beta_c-c_4 a,N}}{Z^\alpha_{\beta_c+c_5a,N}}.
\]
Here $C^c(a)=1$ and $C^f(a)=C'(a)$ was given in \eqref{eq:free part fcns bds varphi in terms of near crit}.
\end{lem}
\begin{proof}
Assume that $A=\{t_0,...,t_k\}$ so that $0=t_0<...<t_k=N$.
We have
\begin{eqnarray*}
\Qset^c_{ \varphi_a,N}(\mathcal{A}^a_N=A/N) &=&
\frac{1}{\ZQ^c_{ \varphi_a,N}}\int_{0}^{a}\cdots \int_{0}^{a} \prod_{i=1}^k f^a_{t_i-t_{i-1}}(y_{i-1},y_i)e^{\varphi_a(y_i)}\dif y_i
=: (*).
\end{eqnarray*}
Using \eqref{eq: lower bound comparison fan fn}, Condition (A), and Lemma \ref{lem:bounding gen fcn of varphi wetting in terms of standard} we have
\begin{eqnarray*}
 (*) &\le& 
\frac{1}{\ZQ^c_{ \varphi_a,N}}e^{(\beta_c+Ca)k}\prod_{i=1}^k f_{t_i-t_{i-1}}\\
&=&
\frac{Z^c_{\beta_c+Ca,N}}{\ZQ^c_{ \varphi_a,N}}\Pset^c_{\beta_c+Ca,N}(\mathcal{A}_N=A/N)\\
&\le& 
\frac{Z^c_{\beta_c+Ca,N}}{Z^c_{\beta_c-C'a,N}}\Pset^c_{\beta_c+Ca,N}(\mathcal{A}_N=A/N).
\end{eqnarray*}
The the lower bound is analogous.
For the free case, fix $A=\{t_0,...,t_k\}$ so that $0=t_0<...<t_k<N$.
\begin{eqnarray*}
\Qset^f_{ \varphi_a,N}(\mathcal{A}^a_N=A/N) &=&
\frac{1}{\ZQ^f_{ \varphi_a,N}}\int_{0}^{a}\cdots \int_{0}^{a} \prod_{i=1}^k f^a_{t_i-t_{i-1}}(y_{i-1},y_i)e^{\varphi_a(y_i)}P^a_{y_k}(N-t_k)\dif y_i
=: (*).
\end{eqnarray*}
Using \eqref{eq: lower bound comparison fan fn}, Condition (A), and Lemma \ref{lem:bounding gen fcn of varphi wetting in terms of standard} we have
\begin{eqnarray*}
 (*) &\le& 
\frac{1}{\ZQ^f_{ \varphi_a,N}} e^{(\beta_c+Ca)(k-1)}\prod_{i=1}^k f_{t_i-t_{i-1}}
\int_{0}^{a} P^a_{y_k}(N-t_k)\varphi_a(y_k)\dif y_k \\
&\le& 
\frac{1}{\ZQ^f_{ \varphi_a,N}} e^{(\beta_c+Ca)k}\prod_{i=1}^k P(N-t_k) f_{t_i-t_{i-1}}\\
&=&
\frac{Z^f_{\beta_c+Ca,N}}{\ZQ^f_{ \varphi_a,N}}\Pset^f_{\beta_c+Ca,N}(\mathcal{A}_N=A/N)\\
&\le& 
\frac{Z^f_{\beta_c+Ca,N}}{C'(a)Z^f_{\beta_c-C'a,N}}\Pset^c_{\beta_c+Ca,N}(\mathcal{A}_N=A/N).
\end{eqnarray*}
Similarly for the lower bound, where we should omit the $C'(a)$ in the analogous statement.

\end{proof}

\section{Near-critical standard wetting, scaling limit of the contact set}\label{sec:near critical standard wetting}
In this section we shall use a result by Julien Sohier on order $1/\sqrt{N}$ near-critical pinning models 
defined by a renewal process with free boundary conditions \cite{sohier2009finite} to deduce that for $o(1/\sqrt{N})$ near-critical 
standard wetting models, and also for pinning models defined by a renewal process with constraint boundary conditions, 
the rescaled limiting contact set coincides with the one which is corresponding to the critical pinning model. 
That is, very roughly speaking, we shall show that in the standard wetting model, the rescaled contact set limit is 
invariant under $o(1/\sqrt{N})$ linear perturbation of the critical pinning strength. 
We now make these statements exact and formal.

First, let us formulate Sohier's result. Let $\tau$ be a renewal process on the positive integers with 
inter-arrival mass function $K$. More precisely, let $\tau_k=\sum_{i=1}^k l_i$ where $l_i$ are i.i.d.\ random variables with 
$\renP(l_1=n)=K(n)$, then $\tau$ is the random subset $\tau: = \{\tau_i:i\ge 0\}\subset \N$ with respect to $\renP$. 
Let $\renE$ be the corresponding expectation. 

Assume that $K(n)=\frac{L(n)}{n^{3/2}}$, where $L$ is slowly varying at infinity (i.e.\ $L(cx)/L(x)\to 1$ as $x\to \infty$ for all $c>0$).

Let $\renP_{\beta,N}$ be a probability measure on subsets of $\{0,...,N\}$ and naturally, on subsets of $\N$, defined by
\[
  \dif \renP_{\beta,N} (\tau)=\dif \renP_{\beta,N} (\tau \cap [0,N]) := \frac{1}{\renZ_{\beta,N}} \exp(\beta |\tau \cap [0,N]|)\dif \renP(\tau)
\]
so that the partition function is $\renZ_{\beta,N}=\renE[\exp(\beta |\tau \cap [0,N]|]$. Let $\renE_{\beta,N}$ be the corresponding expectation.
We also define $\beta_c^{(K)}$ 
by the identity $e^{\beta_c^{(K)}}\sum_{n\ge1}K(n)=1$. Obviously, one notes that $\beta_c^{(K)}=0$ whenever $\sum_{n=1}^\infty K(n) =1$. 

As in Section \ref{sec:main results}, in this section weak convergence of closed random subsets of $[0,1]$ 
is with respect to the Matheron topology on closed subsets.

For readability, we exclude some notations which are irrelevant to our argument and we now formulate a special version of Sohier's theorem.
For elaborated discussion see Sohier \cite[Sections 1 and 3]{sohier2009finite}. See also the monograph \cite{giacomin2007random} for
a comprehensive, rich, and approachable analysis of the renewal model. 

\begin{thm}[Theorem 3.1.(1) and part of the proof of \cite{sohier2009finite} in the case $\alpha=\frac{1}{2}$, $L\sim C_K=e^{\beta_c}$]\label{thm:SohierNearCrit}
Assume $K(n)=q(n)=\frac{C_K}{n^{3/2}}$ so that $\sum_{n\ge 1}K(n)=1$. Let $b=2\sqrt{\pi}C_K$ and fix $\epsilon\in\reals$. 
Then, under $\renP_{\frac{b}{\sqrt{N}}\epsilon,N}$ the rescaled contact set 
$\mathcal{A}_N:=\frac{1}{N}\tau\cap [0,N]:=\{\frac{i}{N}:i\in\tau\cap [0,N]\}\subset[0,1]$ 
is converging weakly to a random set $\mathcal{B}_{1/2}$.
Moreover, the law of $\mathcal{B}_{1/2}$ is absolutely continuous with respect to the law of $\mathcal{A}_{1/2}$, the set of zeros in $[0,1]$ of the standard Brownian motion, 
with Radon-Nikodym density $\frac{\exp(\epsilon L_1)}{\E[\exp(\epsilon L_1)]}$, where $L_1$ is the local time in $0$ of the Brownian motion at time $1$
endowed with probability measure $\pr$ and expectation $\E$.
In particular, for every continuous bounded function $\Phi:\mathcal{F}\to \reals$, where $\mathcal{F}$ is the space of closed sets in $[0,1]$ with the Matheron topology, 
it holds that
\begin{equation}\label{eq:sohier prob limit near crit} 
\renE_{\frac{b\epsilon}{\sqrt{N}},N} [\Phi(\mathcal{A}_N)] =  
\renE\left[\exp\left(b\epsilon\frac{|\tau\cap [0,N]|}{\sqrt{N}}\right)\Phi(\mathcal{A}_N)\right]\to
\E[\exp(\epsilon L_1)\Phi(\mathcal{A}_{1/2})], 
\end{equation}
and specifically
\begin{equation}\label{eq:sohier partition function limit near crit} 
\renZ_{\frac{b\epsilon}{\sqrt{N}},N} =
\renE\left[\exp \left(b\epsilon\frac{|\tau\cap [0,N]|}{\sqrt{N}}\right)\right]
\to \E[\exp(\epsilon L_1)].
\end{equation}
\end{thm}

\begin{rem}
 Following Sohier's notation in lines (3.4) and (3.7) in his paper, in the 
 case $\alpha=\frac{1}{2}$ and $L(x)\sim C_K$, we have $a_n \sim 4\pi C_K^2 n^2$ and $b_n\sim\frac{1}{2\sqrt{\pi}C_K}\sqrt{n}$. 
 We note again that $\beta_c^{(K)}=0$ since $\sum_{n=1}^\infty K(n) =1$.      
 \end{rem}

\begin{rem}
 We note that in the case $K(n)=f_n=\frac{1}{\sqrt{2\pi}}n^{-3/2}$ we have $\beta_c^{(K)}=\beta_c$,
 the critical wetting model pinning strength, and for $K(n)=q(n)=e^{\beta_c} f_n$ we have $\beta_c^{(K)}=0$.
\end{rem}

\begin{cor}\label{cor:Sohier near crit set result holds with decreasing pinning} 
Fix a sequence $\epsilon_N$ so that $\epsilon_N\to 0$ as $N\to\infty$. Let $K(\cdot)=q(\cdot)$, as in 
Theorem \ref{thm:SohierNearCrit}. Then, under $\renP_{\frac{\epsilon_N}{\sqrt{N}},N}$ the rescaled contact set $\mathcal{A}_N$ 
is converging weakly to $\mathcal{A}_{1/2}$, the set of zeros in $[0,1]$ of a standard Brownian motion.
\end{cor}

\begin{proof}
By considering the positive and negative parts of $\epsilon_N$ we may assume WLOG
that they all have the same sign.
We consider the case that they are non-negative. The complementary case is similar. 
First, note that for every $\epsilon>0$ we have by \eqref{eq:sohier partition function limit near crit} that
\[
1\le
\limsup_{N\to\infty} \renZ_{\frac{b\epsilon_N}{\sqrt{N}},N} 
\le 
\lim_{N\to\infty}\renZ_{\frac{b\epsilon}{\sqrt{N}},N}
=
\E[\exp(\epsilon L_1)].
\]
Hence, 
\[
1\le 
\limsup_{N\to\infty} 
\renZ_{\frac{b\epsilon_N}{\sqrt{N}},N} \le 
\liminf_{\epsilon\to 0}
\E[\exp(\epsilon L_1)]
=1
\]
and so 
\begin{equation}\label{eq: prob limit near crit sequence}
\lim_{N\to\infty}\renZ_{\frac{b\epsilon_N}{\sqrt{N}},N}=1.
\end{equation}
More generally, let $\Phi:\mathcal{F}\to \reals$ be a measurable bounded function. 
Considering separately the positive and negative parts in the presentation $\Phi=\Phi_+-\Phi_-$ 
we can assume WLOG that $\Phi$ is non-negative.
For every $\epsilon>0$ we have by \eqref{eq:sohier prob limit near crit} that 
\[
\lim_{N\to\infty} \renE[\Phi(\mathcal{A}_N)]
\le
\limsup_{N\to\infty} \renE\left[\exp\left(b\epsilon\frac{|\tau\cap [0,N]|}{\sqrt{N}}\right)
\Phi(\mathcal{A}_N)\right]
=
\E[\exp(\epsilon L_1)\Phi(\mathcal{A}_{1/2})]. 
\]
Therefore,
\begin{eqnarray*}
\E[\Phi(\mathcal{A}_{1/2})] &=& \lim_{N\to\infty} \renE[\Phi(\mathcal{A}_N)]\\
&\le& \limsup_{N\to\infty} \renE\left[\exp\left(b\epsilon_N\frac{|\tau\cap [0,N]|}{\sqrt{N}}\right)
\Phi(\mathcal{A}_N)\right] \\
&\le&
\liminf_{\epsilon\to 0}\renE[\exp(\epsilon L_1)\Phi(\mathcal{A}_{1/2})]\\
&=& \E[\Phi(\mathcal{A}_{1/2})], 
\end{eqnarray*}
so that  
\begin{equation}\label{eq:partition function limit near crit sequence}
\lim_{N\to\infty} \renE\left[\exp\left(b\epsilon_N\frac{|\tau\cap [0,N]|}{\sqrt{N}}\right)\Phi(\mathcal{A}_{N})\right]=\E[\Phi(\mathcal{A}_{1/2})]. 
\end{equation}
The statement of the corollary follows.
\end{proof}

Define $\renP^c_{\beta,N}$ similarly to be the constrained version of $\renP_{\beta,N}$:
\[
  \renP^c_{\beta,N} (\tau)=\renP_{\beta,N} (\tau \cap [0,N]) := \frac{1}{\renZ^c_{\beta,N}} 
  \exp(\beta|\tau\cap[0,N]|)\one_{\{N\in\tau\}}.
\] 
One can write 
\[
  \renZ^c_{0,N} = \sum_{k=1}^N \renP_{0,N}(\tau_k=N)\renZ_{0,N} = 
  \renE_{0,N}(\one_{\{N\in\tau\}})\renZ_{0,N},
\] 
and so it holds
\[
  \renP^c_{0,N}(\cdot) = \frac{\renP_{0,N} (\cdot)}{\renP_{0,N} (N\in\tau)}
  = \renP_{0,N} (\cdot|N\in\tau)
\] 
(compare with Giacomin \cite[Remark 2.8]{giacomin2007random}).

The next proposition is an analog of Corollary \ref{cor:Sohier near crit set result holds with decreasing pinning} 
in the corresponding constraint case, and moreover for the near-critical \emph{standard wetting model}. 

\begin{prop}\label{prop:near crit standard wetting and renwal scaling limit}
Let $K(\cdot)=q(\cdot)$, as in Theorem \ref{thm:SohierNearCrit}. 
Fix a sequence $\epsilon_N$ so that $\epsilon_N\to 0$ as $N\to\infty$. 
The rescaled contact set $\mathcal{A}_N\subset[0,1]$ distributed according to 
$\Pset^f_{\beta_c+\frac{\epsilon_N}{\sqrt{N}},N}$,
is converging weakly to $\mathcal{A}_{1/2}$, the set of zeros in $[0,1]$ of a standard Brownian motion.
Moreover, when distributed according to either $\renP^c_{\frac{\epsilon_N}{\sqrt{N}},N}$ or $\Pset^c_{\beta_c+\frac{\epsilon_N}{\sqrt{N}},N}$,
$\mathcal{A}_N$ is converging weakly to $\mathcal{A}^c_{1/2}$,
the set of zeros of the Brownian bridge in $[0,1]$.
Here $\renP^c_{\frac{\epsilon_N}{\sqrt{N}},N}$ is corresponding to $K$ with the same conditions as in 
Theorem \ref{thm:SohierNearCrit}, and, as before, all sets are considered in in the Matheron topology on closed subsets of the real line.
\end{prop}
For the proof we shall essentially imitate the way Proposition 5.2.\ of \cite{caravenna2006sharp} was deduced from Lemma 5.3 of that paper (which is partly based on \cite{deuschel2005scaling}), while performing the necessary changes. 
In light of equations \eqref{eq: prob limit near crit sequence} and \eqref{eq:partition function limit near crit sequence} the free case is almost the same as in \cite{caravenna2006sharp}. 
In the constraint cases we will borrow an estimate from \cite{derrida2009fractional}.  
\begin{proof}
First, for the free case, let $A=\{t_1,...,t_{|A|}\}$ so that $0=:t_0<t_1<...<t_{|A|}\le N$.
Note that $\Z_{0,N}=1$ for all $N$ (see \cite[equation (2.17)]{giacomin2007random}), so 
$\renP_{0,N}(A)=\renP(A)$.
Now
\[
\renP(A)=\prod_{j=1}^{|A|}q(t_j-t_{j-1})Q(N-t_{|A|})
\]
where $Q(n)=\bar{K}(n+1)=\sum_{t\ge n+1}q(t)$. 
Also
\[
\Pset^f_{\beta,N}(A)=
\frac{1}{Z^f_{\beta,N}}
e^{(\beta-\beta_c)|A|}P(N-t_{|A|})\prod_{j=1}^{|A|}q(t_j-t_{j-1}),
\]
where as before $P(n)=P^0(n):=\pr[S_1>0,...,S_n>0]$.
We then have for $\beta_N = \beta_c+\frac{\epsilon_N}{\sqrt{N}}$
\[
 \frac{\Pset^f_{\beta_N,N}(A)}{\renP(A)}=\exp\left(\frac{\epsilon_N}{\sqrt{N}}|A|\right)\phi_N(\max A),
\]
where $\phi_N:[0,1]\to \reals_+$ is defined by
\[
 \phi_N(t):=\frac{1}{Z^f_{\beta_c,N}}\frac{P(N((1-t))}{Q(N(1-t))}.
\]
Therefore for every bounded measurable functional $\Phi$ we have
\[
 \Eset^f_{\beta_N,N}[\Phi(\mathcal{A}_N)]=
 \renE\left[\exp\left(\frac{\epsilon_N}{\sqrt{N}}|\mathcal{A}_N|\right) \phi_N(\max A)\Phi(\mathcal{A}_N)\right],
\]
It was proved in \cite[proof of Proposition 5.2.]{caravenna2006sharp} that 
$\phi_N(t)\to 1$ uniformly in $t\in[0,v]$, for every $v\in(0,1)$.
Since $\pr$-a.s.\ $0\notin\mathcal{A}_{1/2}$, it follows from \eqref{eq:partition function limit near crit sequence} (for general $\epsilon_N\to 0$)  that   
\[
\renE\left[\exp\left(\frac{\epsilon_N}{\sqrt{N}}|\mathcal{A}_N|\right) \phi_N(\max A)\Phi(\mathcal{A}_N)\right]\to \E[\Phi(\mathcal{A}_{1/2})],
\]
and the free case is done.
We will now show the constraint case. 
By definition, for every $A\subset \{1,...,N\}$ containing $N$ we have 
\[
\frac{\renP^c_{\beta,N}(A)}
{\Pset^c_{\beta_c+\beta,N}(A)}=
\frac{Z^c_{\beta_c+\beta,N}}
{\renZ^c_{\beta_c,N}}.
\]
That is, the ratio of these to probability measures is constant and so they coincide. 
We shall work with $\renP^c_{\frac{\epsilon_N}{\sqrt{N}},N}$.
As in the free case we follow the proof of \cite[Proposition 5.2.]{caravenna2006sharp}, and accordingly  
we now consider $\mathcal{A}_N\cap[0,1/2]$. We have for $\beta_N = \beta_c+\frac{\epsilon_N}{\sqrt{N}}$
\[
 \renE^c_{\beta_N,N}[\Phi(\mathcal{A}_N\cap[0,1/2])]=
 \renE\left[\exp\left(\frac{\epsilon_N}{\sqrt{N}}|\mathcal{A}_{N}\cap[0,1/2]|\right) \phi_N^c(\max \mathcal{A}_{N}\cap[0,1/2])\Phi(\mathcal{A}_N\cap[0,1/2])\right],
\]
where 
\[
\phi_N^c(t):=\frac{\sum_{n=0}^{N/2}\renZ^c_{\frac{\epsilon_n}{\sqrt{n}},n}q(N(1-t)-n)}
{\renZ^c_{\frac{\epsilon_N}{\sqrt{N}},N}Q(N(1-t))}, \text{   } t\in[0,1/2]. 
\]
We remind the reader that here $\renZ^c_{\beta,N}$ is the partition function corresponding to 
$\renP^c_{\beta,N}$.
Now, since $\phi_N^c(t)$ is defined similarly to $f_N^c(t)$ in the proof of \cite[Proposition 5.2.]{caravenna2006sharp},
with the only difference being that all the $\renZ^c_{\frac{\epsilon_k}{\sqrt{k}},k}$ are replaced by the corresponding $\renZ^c_{0,k}$,
and since that proof uses only the asymptotic rates of $\renZ^c_{0,\cdot}$,$q(\cdot)$ and $Q(\cdot)$,  
we are done once we show that 
\begin{equation}\label{eq:constraint partition near crit goes to crit}
 \frac{\renZ^c_{\frac{\epsilon_N}{\sqrt{N}},N}}{\renZ^c_{0,N}}\to 1 \text{ as } N\to\infty.
\end{equation}
By a direct expansion, one finds that
$\renZ^c_{\frac{\epsilon_N}{\sqrt{N}},N}=
\renZ^c_{0,N}\renE^c_{0,N}\left[\exp\left(\frac{\epsilon_N}{\sqrt{N}}|\tau\cap[0,N]|\right)\right]$.
Therefore,
\[
\frac{\renZ^c_{\frac{\epsilon_N}{\sqrt{N}},N}}{\renZ^c_{0,N}}= 
\renE^c_{0,N}\left[\exp\left(\frac{\epsilon_N}{\sqrt{N}}|\tau\cap[0,N]|\right)\right]=
\renE\left[\exp\left(\frac{\epsilon_N}{\sqrt{N}}|\tau\cap[0,N]|\right)|N\in\tau\right].
\]
Assume WLOG that $\epsilon_N\ge0$ for all $N$ and fix $\epsilon>0$. Since for large $N$ 
the right most expression in last line is smaller than   
$\renE\left[\exp\left(\frac{\epsilon}{\sqrt{N}}|\tau\cap[0,N]|\right)|N\in\tau\right]$,
by \cite[equation (A.12)]{derrida2009fractional} (cf.\ \cite{toninelli2009coarse}, and \cite[Lemma A.2]{giacomin2010marginal}), 
there is a constant $C>0$ bounding the expression.
Using Lemma \ref{lem:enough to bound} we deduce that the expression is in 
fact converging to $1$ as $N\to\infty$, and so we have \eqref{eq:constraint partition near crit goes to crit}. 
We therefore conclude the proof of the proposition.
\end{proof}

\section{Contact set scaling limit - proof of Theorem \ref{thm:main thm general zero set limit} }\label{sec:pf of scaling contact set}

First, we note that for $a_N=\epsilon_N/\sqrt{N}$, for $s,r\in\reals$, and positive functions $C(a)$ 
converging to $1$ as $a\to0$ we have by \eqref{eq:constraint partition near crit goes to crit} that
\[
\frac{Z^c_{\beta_c+r a_N,N}}{Z^c_{\beta_c+s a_N,N}C(a_N)}\to 1.
\]
Moreover, by \eqref{eq: prob limit near crit sequence} we have
\[
Z^f_{\beta_c,N}\to 1 \text{ and } Z^f_{\beta_c+r a_N,N}\to 1 \text{ as } a\to 0.
\]
Next, using Proposition \ref{prop:near crit standard wetting and renwal scaling limit}
with $r \epsilon_N$ instead of 
$\epsilon_N$ we have the desired corresponding scaling limit under $\Pset^\alpha_{\beta_c+r a_N}$.
Using Lemma \ref{lem:derivative varphi wetting to near crit standard} we can now conclude.
Indeed, let $\Phi:\mathcal{F}\to \reals$ be a measurable bounded function. 
As before, considering separately the positive and negative parts in the presentation $\Phi=\Phi_+-\Phi_-$ 
we can assume WLOG that $\Phi$ is non-negative.
We therefore have by Lemma \ref{lem:derivative varphi wetting to near crit standard}
\[
 \EQset^\alpha_{ \varphi_a,N}[\Phi(\mathcal{A}_N)] \le 
 R_N \Eset^\alpha_{\beta_c +c_3a_N,N}[\Phi(\mathcal{A}_N)]\to 
 \E[\Phi(\mathcal{A}^\alpha_{1/2})]   
\]
and
\[
 \EQset^\alpha_{ \varphi_a,N}[\Phi(\mathcal{A}_N)] \ge 
 L_N \Eset^\alpha_{\beta_c - c_ 6 a_N,N}[\Phi(\mathcal{A}_N)]\to 
 \E[\Phi(\mathcal{A}^\alpha_{1/2})],   
\]
where $L_N, R_N$ are positive constants so that $L_N,R_N\to 1 $.

\section{Path scaling limit - proof of Theorem \ref{thm:main thm general path limit}}\label{sec:Path scaling limit}
In this section we shall prove Theorem \ref{thm:main thm general path limit}.
Once we have the contact set convergence, Theorem \ref{thm:main thm general zero set limit}, to move 
to the path limit is by now routine, following the guidelines of \cite{deuschel2005scaling}.
We shall highlight the necessary modifications. Let us first give a rough sketch.

Tightness will be proved as in \cite[Lemma 4]{deuschel2005scaling} where we need a small linear 
modification of the oscillation function, and instead of using Propositions 7 and 8 of that paper, 
we shall use stronger results as follows.   
The first result is the weak convergence in $C[0,1]$ under $\Pset^c_{0,N}(x_N,y_N)$ 
the pinning-free process (i.e.\ $\varphi_{a}=0$) conditioned on the starting and 
ending point $x_N,y_N\in[0,a_N]$ to the Brownian bridge, which was proved by Caravenna-Chaumont \cite{caravenna2013invariance}.
The second result is the analogous statement on the free case and the Brownian meander which is available by 
Caravenna-Chaumont \cite{caravenna2008invariance}. 

Once we have tightness, we need to prove the finite dimensional distributions, for that we follow \cite[Chapter 8]{deuschel2005scaling}.
Since we know that out contact set converges to the zero-set of the Brownian motion or bridge, then we know that 
the probability that a fixed finite number of points in $[0,1]$ are the limiting zero-set is $0$, and there is no 
change of that part of the argument. The only difference in the proof is that we condition not only on the contact indices 
but also on \emph{their location in the strip}. But since the conditioned processes converge 
by the last two cited theorems, we can conclude using dominated convergence on the full path as in \cite{deuschel2005scaling}.

Let $A_n^a(y):=\{S_1>a,...,S_{n-1}>a,S_n=y\}$. We have the following densities comparison bound. 
\begin{lem}\label{lem:osillation comparison}
 For every $\gamma>0$ and $n\in\N$, we have
 \begin{equation}
  \pr_x\left( \max_{0\le i,j\le n}|S_i-S_j|>\gamma ,{A_n^a(y)}\right)
  \le
  \pr_0\left( \max_{0\le i,j\le n}|S_i-S_j|>\gamma -a ,{A_n^0(0)}\right)
 \end{equation}
uniformly in $x,y\in[0,a]$. 
Moreover, the same holds whenever in both sides of the inequality 
the index set satisfies in addition that 
$|i-j|\le \delta n$ for some fixed $\delta>0$.
\end{lem}
\begin{proof}
Let $a-x=S_0,S_1,...,S_n=a-y$ so that $S_i\ge 0$, $i=1,...,n-1$, and 
$|S_{i_0}-S_{j_0}|= \max_{0\le i,j\le n}|S_i-S_j|$. 
Then, if $i_0,j_0\notin \{1,...,n-1\}$, WLOG $i_0=0$,  and so 
$|S_{i_0}-S_{j_0}|= |S_{j_0}-(a-x)|
\le |S_{j_0}| + |a-x| 
\le |S_{j_0} - 0| + a$.
In other words, $\max_{0\le i,j\le n}|S_i-S_j|\le \max_{0\le i,j\le n}|S_i'-S_j'| + a$ where 
$S_i'=S_i$ for $i=1,...,n-1$ but $S_0=S_n=0$. Therefore,
\begin{eqnarray*}
{\pr_x\left( \max_{0\le i,j\le n}|S_i-S_j|>\gamma ,{A_n^a(y)}\right)}
  &=&
  \frac{1}{\kappa^n}
\int_a^{\infty}...\int_a^{\infty} \one_{\max_{i,j\le n}|S_i-S_j|>\gamma}\times\\
 &\times& \rho(s_1-x)\rho(s_2-s_1)\cdots \rho(s_{n-1}-s_{n-2})\rho(y-s_{n-1})\dif s_1\cdots\dif s_{n-1}\\ 
  &=&
  \frac{1}{\kappa^n}
\int_0^{\infty}...\int_0^{\infty} \one_{\max_{0\le i,j\le n}|S_i-S_j|>\gamma}\times\\
 &\times& \rho(s_1-x+a)\rho(s_2-s_1)\cdots \rho(s_{n-1}-s_{n-2})\rho(y-s_{n-1}-a)\dif s_1\cdots\dif s_{n-1}\\ 
 &\le&  
   \frac{1}{\kappa^n}
\int_0^{\infty}...\int_0^{\infty} \one_{\max_{0\le i,j\le n}|S_i-S_j|>\gamma-a}\times\\
 &\times& \rho(s_1)\rho(s_2-s_1)\cdots \rho(s_{n-1}-s_{n-2})\rho(s_{n-1})\dif s_1\cdots\dif s_{n-1}\\ 
  &=& \pr_0\left( \max_{0 \le i,j\le n}|S_i-S_j|>\gamma -a ,{A_n^0(0)}\right).
\end{eqnarray*}
The `moreover' part is similar, we omit its proof. 
\end{proof}

We shall now prove that whenever $\varphi=\varphi_{a_N}^0$, i.e. no pinning is present, the scaling limit is 
a Brownian excursion, for any fixed endpoints $x_N,y_N\in[0,a_N]$. Shifting by $a_N$, it is 
equivalent to show that conditioning on starting and ending at $S_0=x_N-a_N,S_N=y_N-a_N$ and $S_n$ non-negative at 
times $1\le n\le N-1$, the rescaled path converges weakly to the Brownian excursion.

The following is a formulation of Theorem 1.1 of Caravenna-Chaumont \cite{caravenna2013invariance} which shows
the same for non-negative endpoints which are $o(\sqrt{N})$ away from the zero line. 
Our modification will follow by comparison tightness and finite dimensional distributions with Caravenna-Chaumont.

Let us first introduce a notation for the conditioning. Define 
\[
\pr_{x,y}^{+,N}:=\pr_x(\cdot|\mathcal{C}_{N-1}, S_N=y), 
\]
for any $x,y\in\reals$, $N\in\N$.
\begin{thm}[Caravenna-Chaumont \cite{caravenna2013invariance}]\label{thm:Caravenna-Chaumont conditioned bridges}
 Let $(x_N),(y_N)$ be sequences of non-negative real numbers such that $x_N,y_N=o(\sqrt{N})$ as $N \to \infty$. 
 Then under $\pr_{x_N,y_N}^{+,N}$, $(X^{(N)}_t)_{t\in[0,1]}$
 converges weakly in $C[0,1]$ to the Brownian excursion.  
\end{thm}
We will formulate the next theorem in a somewhat non-elegant way, but it will be helpful for us later-on.
\begin{thm}\label{thm:conditioned bridges with negative endpoints}
Let $(x_N),(y_N)$ be sequences of non-negative real numbers such that $x_N,y_N\le a_N=o(\sqrt{N})$ as $N \to \infty$. 
 Then under $\pr_{x_N-a_N,y_N-a_N}^{+,N}$, $(X^{(N)}_t)_{t\in[0,1]}$
 converges weakly in $C[0,1]$ to the Brownian excursion.  
\end{thm}
We note that the assumption $x_N,y_N\le a_N=o(\sqrt{N})$ is only to make sure that $X^N_0,X^N_1\to 0$.
We will use the theorem later on with a much stronger condition $a_N=o(\frac{1}{\sqrt{N}})$.
\begin{proof}
First we prove tightness. For a path $x\in C[0,1]$ define 
\begin{equation}\label{eq:def of oscillation fcn Gamma}
\Gamma(\delta)(x) := \sup_{\{t,s\in[0,1]:|t-s|\le\delta\}}|x_t-x_s|. 
\end{equation}

Using the fact $f_N^0(x_N-a_N,y_N-a_N)=f_N^{a_N}(x_N,y_N)$ let us rewrite
Lemma \ref{lem:osillation comparison}: 
  \begin{equation*}
  \pr_{x_N-a_N,y_N-a_N}^{+,N}\left( \max_{|i-j|\le \delta n}|S_i-S_j|>\gamma\right)f_N^{a_N}(x_N,y_N)
  \le
\pr_{0,0}^{+,N}\left( \max_{|i-j|\le \delta n}|S_i-S_j|>\gamma - a_N \right)f_N^0(0,0)
 \end{equation*}
 For every $\delta, \gamma>0$ and $n\in\N$, uniformly in $x_N,y_N\in[0,a_N]$. 
Now, by \eqref{eq: upper bound fan fn} and \eqref{eq: lower bound comparison fan fn} we get
\begin{equation}\label{eq:tightness}
  \pr_{x_N-a_N,y_N-a_N}^{+,N}\left( \max_{|i-j|\le \delta n}|S_i-S_j|>\gamma\right)
  \le
\exp(C_0 a_N) \pr_{0,0}^{+,N}\left( \max_{|i-j|\le \delta n}|S_i-S_j|>\gamma - a_N \right).
 \end{equation}
Caravenna-Chaumont Theorem \ref{thm:Caravenna-Chaumont conditioned bridges} implies in particular that 
$(X^{(N)}_t)_{t\in[0,1]}$ is tight under $\pr_{0,0}^{+,N}$,
and so by \eqref{eq:tightness}, it is also tight under $\pr_{x_N-a_N,y_N-a_N}^{+,N}$.
Indeed, the standard necessary and sufficient condition for tightness on
$C[0,1]$ is Ascoli-Arzel\'a and Prokhorov Theorems: for every $\gamma > 0$
$\lim_{\delta\to0}\sup_N \pr_{0,0}^{+,N} (\Gamma(\delta)>\gamma)=0$. 
To get our tightness, fix $\gamma>0$.
Choose $N_0$ large enough so that $\gamma-a_N>\gamma/2$ for all $N\ge N_0$. Tightness will 
hold by considering only $\delta<1/{N_0}$.

We shall now prove convergence of the finite dimensional distributions. 
Let $0<s_1<...<s_n<1$. Fix $N$ large enough so that $1/N<s_1<s_n<1-1/N$. Then $(X^{(N)}_{s_i})_{i=1,...,n}$ 
have the same distribution under both conditional 
distributions $\pr_{0,0}^{+,N}(\cdot|S_1=x,S_{N-1}=y)$ and 
$\pr_{x_N-a_N,y_N-a_N}^{+,N}(\cdot|S_1=x,S_{N-1}=y)$, for all $x,y\ge0$. 
Since $\frac{x_N}{\sqrt{N}},\frac{y_N}{\sqrt{N}}\to 0$, the difference 
between the corresponding expectations on any test function on $(X^{(N)}_{s_i})_{i=1,...,n}$ 
goes to zero as $N\to \infty$. 
We conclude by the convergence of the distributions of $(X^{(N)}_{s_i})_{i=1,...,n}$ under  
$\pr_{0,0}^{+,N}$, 
using Caravenna-Chaumont Theorem \ref{thm:Caravenna-Chaumont conditioned bridges} again.
\end{proof}

\begin{proof}[proof of Theorem \ref{thm:main thm general path limit}]
First, we shall prove tightness of $\left((X^{(N)}_t)_{t\in[0,1]},\pr^{\alpha}_{\varphi_{a_N},N}\right)$.
 We modify the definition \eqref{eq:def of oscillation fcn Gamma} as follows.  
 For a path $x\in C[0,1]$ define the modified $\delta$-oscillation of strip size $a$ by 
\begin{equation}\label{eq:def of oscillation fcn Gamma modified}
\tilde {\Gamma}^a(\delta)(x) := \sup_{\{t,s\in[0,1]:|t-s|\le\delta, s\sim_x t\}}|x_t-x_s|,
\end{equation}
where $s\sim_x t$ if and only if $x_u>a$ for all $u\in (s,t)$
 (see \cite{caravenna2007tightness} for the case $a=0$). 
 The next lemma shows bounds the oscillations on the $\varphi_a$-model 
 conditioned on the contact set and the contact locations (!) by the standard models oscillations
conditioned on the contact set. For ease of notation we denote by $i\sim_N j$ whenever 
$\frac{i}{N} \sim_{X^{(N)}} \frac{j}{N}$.
 \begin{lem}
$\pr^\alpha_{\varphi_{a},N}(\max_{|i-j|\le \delta N, i\sim_N j}|S_i-S_j|>\gamma | A, y_1,...,y_{|A|})\le
 \exp(C_0 a|A|)\pr^\alpha_{\beta_c,N} (\max_{|i-j|\le \delta N, i\sim_N j}|S_i-S_j|>\gamma-a | A)$
where $A$ is the contact set, $y_i\in[0,a]$ are the corresponding values in the strip. 
 \end{lem}
\begin{proof}
Note that conditioning on $A$ the excursions are independent. Moreover, 
conditioning on the endpoints the law of the excursions is the same as with respect to 
$\pr^{+,N}_{y_{i-1},y_i}$. By iterating \eqref{eq:tightness} $|A|$ times we conclude.
\end{proof}

\begin{cor}
If $a_N=o(N^{-1/2})$ then the sequence $\left((X^{(N)}_t)_{t\in[0,1]},\pr^{\alpha}_{\varphi_{a_N},N}\right)$
is tight.
\end{cor}
\begin{proof}
First, we naturally extend the definition of $\Qset^\alpha_{\varphi_{a},N}$ to include pairs $(A,y)$ 
where $y\in[0,a]^|A|$ the vector of positions at the contact indices. 
Since $\Gamma(\delta)(x) \le \tilde\Gamma^a(\delta)(x)$, it is enough to show that
$\pr^\alpha_{\varphi_{a},N}(\tilde {\Gamma}^a(\delta)(x)>\gamma) \to 0$ as $\delta \to 0$.
\begin{eqnarray*}
 \pr^\alpha_{\varphi_{a},N}(\tilde {\Gamma}^a(\delta)>\gamma) &=& 
 \sum_{A\subset\{0,...,N\}} \int_{0}^a ...\int_{0}^a
 \pr^\alpha_{\varphi_{a},N}(\tilde {\Gamma}^a(\delta)>\gamma|A,y_1,...,y_{|A|})\times \\
 & & \times 
 \Qset^\alpha_{\varphi_{a},N}(A,y_1,...,y_{|A|})\dif y_1\cdots \dif y_{|A|} \\ 
 &\le&
 \sum_{A\subset\{0,...,N\}} 
 \exp(C_0 a_N|A|) \pr^\alpha_{\beta_c,N}(\tilde {\Gamma}(\delta)>\gamma-a_N|A)
  \Qset^\alpha_{\varphi_{a},N}(A) \\ 
\end{eqnarray*}
Now, from Lemma \ref{lem:derivative varphi wetting to near crit standard}, using the fact that $a_N\to 0$, 
we have $C'_N\to1$ so that
\begin{eqnarray*}
 {\Qset^\alpha_{\varphi_{a},N}(A)} &\le& 
 C'_N  {\Pset^\alpha_{\beta_c+c_3 a_{N},N}(A)})
\end{eqnarray*}
The partition functions ratio between pinning perturbation of constant times $a_N$ is going to 1. Hence 
we have
\begin{eqnarray*}
 \sum_{A\subset\{0,...,N\}} 
 \exp(C_0 a_N|A|) \pr^\alpha_{\beta_c,N}(\tilde {\Gamma}(\delta)>\gamma-a_N|A)
  \Qset^\alpha_{\varphi_{a},N}(A)  \\
  \le
 C_N \sum_{A\subset\{0,...,N\}} 
 \pr^\alpha_{\beta_c,N}(\tilde {\Gamma}(\delta)>\gamma-a_N|A)
 \Pset^\alpha_{\beta_c+(c_3+C_0) a_{N},N}(A)
\end{eqnarray*}
for some $C_N\to1$.
To end, note that the conditioning allows us to change $\beta_c$, to get 
\begin{eqnarray*}
 C_N \sum_{A\subset\{0,...,N\}} 
 \pr^\alpha_{\beta_c,N}(\tilde {\Gamma}(\delta)>\gamma-a_N|A)
 \Pset^\alpha_{\beta_c+(c_3+C_0) a_{N},N}(A) \\
\le
\tilde C_N \sum_{A\subset\{0,...,N\}} 
 \pr^\alpha_{\beta_c+(c_3+C_0) a_{N},N}(\tilde {\Gamma}(\delta)>\gamma-a_N|A)
 \Pset^\alpha_{\beta_c+(c_3+C_0) a_{N},N}(A)\\
 =
 \tilde C_N \pr^\alpha_{\beta_c+(c_3+C_0) a_{N},N}(\tilde {\Gamma}(\delta)>\gamma-a_N)
 \end{eqnarray*}

To sum up, tightness follows once we show tightness under $\pr^\alpha_{\beta_c+(c_3+C_0) a_{N},N}$. 
The latter is a special case of \cite[Proposition 4.3]{caravenna2018maximum}.
 \end{proof}

To prove the convergence of finite dimensional distributions we follow closely \cite[Chapter 8]{deuschel2005scaling},
with the necessary modifications. Let us deal with the constraint case. Let $(\beta_t)_{t\in[0,1]}$ be the Brownian bridge. Let $0<s_1<...<s_n<1$. 
Remember the law of $\mathcal{A}^\alpha_N$ given in \eqref{eq:def of measure on sets with varphi}, where 
$\varphi_a$ satisfying Condition $A$.

To unify the notations denote by $Z(x)$ the zero-set of the path $x\in C[0,1]$. 
Given a closed set $Z\subset[0,1]$ and $t\in[0,1]$ we let $d_t(Z):=\inf Z \cap [t,1]$, 
$g_t(Z):=\sup Z \cap [0,t]$, and $\Lambda_t(Z):=d_t-g_t$. 

By Theorem \ref{thm:main thm general zero set limit} and the Skorokhod representation Theorem there is
a sequence $Z_N$ with laws $\mathcal{A}^c_N$ converging a.s. to $\mathcal{A}_{1/2}^c$, 
in the Matheron topology defined above.

We define random equivalence relations, with respect to $Z_N$, 
on $\{s_1,...,s_n\}$ by declaring that $s_i\sim s_j$ 
if and only if either $d_{s_i}=d_{s_j}$ or $g_{s_i}=g_{s_j}$. 
In words, $s_i\sim s_j$ if and only if $(s_i,s_j)$ is contained in an excursion of $X^{(N)}$ (in law).

Notice that a.s.\ $(\beta_{s_i})\ne 0$ for all $1\le i\le n$. 
Since the Matheron topology is also homehomorphic to the Hausdorf metric space (see (29) and (30) in \cite{deuschel2005scaling})
then $g_{s_i}(Z_N)$ and $d_{s_i}(Z_N)$ converge a.s.\ to strictly positive random variables, and $A_k^N,k=1,...,I^N$,
the random equivalent classes of $\{s_1,...,s_n\}$  (here $I^N\le n$)
are a.s.\ eventually constant with $N$ (but still random). Denote it by $A_k,k=1,...,I$.
Let $W_{s_i}^{N,(y^N_{i-1},y^N_i)}$ $i=1,...,n$, $y_i^N\in[0,a_N]$ 
be a set of random variables with values in $C[0,1]$, so that 
$W_{s_i}^{N,(y^N_{i-1},y^N_i)}$ is distributed as $X^N$ under $\pr^{+,N}_{y_{i-1},y_i}$, and is independent
of $g_{s_i}(Z_N)$ and $\Lambda_{s_i}(Z_N)$.
Theorem \ref{thm:conditioned bridges with negative endpoints} tells us that $W_{s_i}^{N,(y^N_{i-1},y^N_i)}$ converges weakly
to the Brownian excursion $(\mathcal{E}_t)_{t\in[0,1]}$.
Set 
\[
M^N_{s_i} = \sum_{k=1}^{I^N}\one_{s_i\in A_k^N}\sqrt{\Lambda_{A^N_k}}W_{s_i}^{N,(y^N_{i-1},y^N_i)}
\left( \frac{s_i-g_{A_k}}{\Lambda_{A_k}}
\right).
\]
Then $(M_{s_i}^N)_{i=1,...,n}$ is distributed at $\pr^c_{\varphi_{a_N},N}$ conditioned on the excursions' endpoints ${y_1,...,y_{I^N}}$.
Noting that the density 
$\pr((|\beta_{s_i}|)_{i\in A_k}\in\dif x)=\pr( \sqrt(\Lambda_{A_k}) (\mathcal{E}_{s_i/\Lambda_{A_k}})_{i\in A_k}\in\dif x)$ for any $x\in\reals^{A_k}$
(see \cite[Chapter 8]{deuschel2005scaling}). 
Using dominated convergence and the Brownian scaling of $(\mathcal{E}_t)_{t\in[0,1]}$, the 
finite dimensional distributions for the path conditioned on the 
endpoints $y_i^N$ has a limiting law $|\beta|$. 
But since the limit is independent of $y_i^N$, we conclude.  
The free case follows analogously. 
\end{proof}


\section{The strip wetting model with constant pinning}\label{sec:proofs strip wetting model}
The goal in this chapter is to prove Theorem \ref{thm:critical wetting close to standard critical wetting}.
\subsection{The associated Markov renewal process, integral operator, and free energy, and the critical value}
\label{sec:Detour - the associated integral operator and the critical value derivation}
To fix notations and for sake of self containment, we shall elaborate on 
the analysis of the strip wetting model, and follow closely Sohier \cite{sohier2015scaling}. 
We state here the argument mostly without proofs, which can be found in \cite{sohier2015scaling}.
We remind the reader that in our case $\varphi=\varphi_a^\beta:=\beta\one_{[0,a]}$. 
Here $a\ge 0$ and $\beta\in\reals$ are the corresponding parameters. 
Let us first introduce a notation for the corresponding measures in this case.
\begin{equation}\label{eq:StripModelDefFree}
{\dif{\pr}^f_{a,\beta,N}} (S)= \frac{1}{{Z}^f_{a,\beta,N}}
 \exp\left(\beta\sum_{k=1}^N\one_{[0,a]}(S_k)\right)\one_{\mathcal{C}_N}{\dif\pr_{0}}(S),
\end{equation}
\begin{equation}\label{eq:StripModelDefConstraint}
{\dif{\pr}^c_{a,\beta,N}} (S)= \frac{1}{{Z}^c_{a,\beta,N}}
 \exp\left(\beta\sum_{k=1}^N\one_{[0,a]}(S_k)\right)\one_{[0,a]}(S_{N})\one_{\mathcal{C}_N}{\dif\pr_{0}}(S),
\end{equation}
and the density
\begin{equation}\label{eq:StripModelDefConstraintDensity}
{{Z}^c_{a,\beta,N}} (S)(x,y)= \E_{x}\left[
 \exp\left(\beta\sum_{k=1}^N\one_{[0,a]}(S_k)\right)\one_{\mathcal{C}_N} \one_{\{ y\}}(S_{N}) \right].
\end{equation}
Remember the density
$$f^a_n(x,y):=\frac{1}{\dif y}\pr _x [S_1>a,...,S_{n-1}>a,S_n\in \dif y]$$
 with respect to the Lebesgue measure,
where $$f^a_1(x,y):=\rho(x-y).$$
Define the resolvent kernel density on $[0,a]$
\begin{equation}
b_\lambda^a(x,y):= \sum_{n=1}^\infty e^{-\lambda n} f^a_n(x,y)\one_{[0,a]^2}(x,y)
\end{equation}
for all $\lambda\ge 0$.
The following Lemma is an easy estimate, we differ its proof to Appendix \ref{appendix:app1}.
\begin{lem}\label{lem:b is Hilbert-Schmidt}
$b_\lambda^a$ is a kernel density of a Hilbert-Schmidt integral operator, for all $\lambda\ge0$.
In other words,
$\int_{0}^{\infty} \int_{0}^{\infty} b^a_\lambda(x,y)^2 \dif x\dif y <\infty$.
\end{lem}
Let $\delta_a(\lambda)$ be the eigenvalue corresponding to the integral operator defined by the kernel density $b_\lambda^a$. We note that since $b_\lambda^a$ is smooth,
strictly positive, and point-wise decreasing with $\lambda\ge 0$, then $\delta_a(\lambda)$ is also decreasing,
continuous and moreover, its corresponding left eigenfunction $V^a_{\lambda}(\cdot)$ is continuous and strictly positive on $[0,a]$.
In particular, $\delta_a(\lambda)$ has an function inverse which is also continuous, strictly positive and decreasing
$\delta_a^{-1}(\cdot):[0,\delta_a(0))\to (0,\infty)$.

Define the free energy by
$$F^a(\beta):= \delta_a^{-1}(e^{-\beta})$$ whenever $\beta\ge\beta_c(a):=-\log(\delta_a(0))$ and set
$F^a(\beta):= 0$ if $\beta<\beta_c(a)$.
In the critical and supper-critical case, $\beta\ge\beta_c(a)$, we denote the corresponding left eigenfunction by
$V_{a,\beta}(\cdot):=V^a_{F^a(\beta)}(\cdot)$, that is left eigenfunction equation reads:
\begin{equation}
\int_{0}^a \sum_{n=1}^\infty e^{-F^{a}(\beta) n} f^a_n(x,y) \frac{V_{a,\beta}(y)}{V_{a,\beta}(x)}e^\beta\dif y =1
 \end{equation}
for all $x\in[0,1]$. 
Note that by symmetry of $f_n^a$, the left eigenvalue equals the right eigenvalue and moreover one can check that 
in this case the measure with density $V_{a,\beta}^2$ is invariant for the Markov process on $[0,a]$
with jump density $\int_{0}^a \sum_{n=1}^\infty e^{-F^{a}(\beta) n} f^a_n(x,y) \frac{V_{a,\beta}(y)}{V_{a,\beta}(x)}e^\beta$.

In the critical case we omit the $\beta_c(a)$ from the notation and write  
$$V_{a}(\cdot):=V^a_{F^a(\beta_c)}(\cdot)=V^a_{0}(\cdot).$$
In particular,
\begin{equation}
\int_{0}^a \sum_{n=0}^\infty q^a_n(x,y) \dif y =1
 \end{equation}
 for all $x\in[0,a]$, where $ q^a_n(x,y) := \frac{1}{\gamma^a(x,y)}f^a_n(x,y) = f^a_n(x,y) \frac{V_{a}(y)}{V_{a}(x)}e^{\beta_c(a)}$.

\subsection*{Strip model in terms of Markov renewal}\label{sec:markov renewal}
Let $\mathcal{P}^\beta$ be measure of a Markov renewal process $(\tau, J)$ on $\N\times [0,a]$ with kernel
density \[q^{a,\beta}_n(x,y):=e^{-F^{a}(\beta) n} f^a_n(x,y) \frac{V_{a,\beta}(y)}{V_{a,\beta}(x)}e^\beta.
\]
In particular, at criticality $q^{a,\beta_c(a)}=q^{a}$.
We then have
\[
{Z}^c_{a,\beta,N}(x,y)\dif y=\mathcal{P}^\beta(N\in \tau, j_0=x,j_N\in\dif y)e^{F^a(\beta)N} \frac{V_{a,\beta}(x)}{V_{a,\beta}(y)}.
\]
And in particular
\[
{Z}^c_{a,\beta_c(a),N}(x,y)\dif y=\mathcal{P}^\beta(N\in \tau, j_0=x,j_N\in\dif y)\frac{V_{a}(x)}{V_{a}(y)}.
\]

Therefore, under our initial measure the density of the zero-set $A$ in $[0,N]$ together with the corresponding points 
$J(A)\subset [0,a]^{|A|}$ is
\[{\pr}^c_{a,\beta,N}((A,J(A)))= \mathcal{P}^\beta((A,J(A))|N\in \tau),\]
and more generally
\[{\pr}^c_{a,\beta,N}(x,y)((A,J(A)))= \mathcal{P}^\beta((A,J(A))|N\in \tau, j_0=x, j_N=y).\]

\subsection{Strip wetting with critical pinning satisfies Condition (A) - proof of Theorem \ref{thm:critical wetting close to standard critical wetting}.}\label{sec:Pf of Proposition critical wetting close to standard critical wetting}
We will choose $V_a$ so that $\int_0^a V_a(x)^2\dif x =1$. Remember the eigenvalue equation 
\[
V_{a}(x) = e^{\beta_c(a)} \int_{0}^a \sum_{n\ge 1} f^a_n(x,y)V_{a}(y)\dif y,
\]
$x\in [0,a]$.
Note that for a fixed $a>0$, $V_{a}$ is continuous and strictly positive on $[0,a]$ since does $f^a_n(x,y)$. Also 
$f^a_n$ is continuous and is dominating a summable series (of the form $c(a)n^{-3/2}$), then so is $V_{a}$, and moreover its
derivatives, whenever defined, are given by
\[
\frac{\partial^m}{\partial x^m} V_{a}(x) =e^{\beta_c(a)}\int_{0}^a \sum_{n\ge 1}\frac{\partial^m}{\partial x^m}f^a_n(x,y)
V_{a}(y)\dif y,
\]
$m\ge 1$.
Therefore, the simple estimate $\frac{\partial}{\partial x}f^a_n(x,y)\ge (a-x) \ge f^a_n(x,y)$ implies that also
\[
 \frac{\partial}{\partial x}V_a(x)\ge (a-x)  V_a(x).
 \]
Integrating, we get
\begin{equation}\label{eq:estimate V}
\frac{V_a(z)}{V_a(x)} \ge e^{a(z-x)-\frac{1}{2}(z^2-x^2)}
\end{equation}
whenever $0\le x\le z \le a$.
Using it for $z=a, x=y$, we have
\begin{eqnarray}
e^{-\beta_c(a)} &=& \int_{0}^a \sum_{n\ge 1} f^a_n(a,y)\frac{V_{a}(y)}{V_{a}(a)}\dif y\\
&\le&
\int_{0}^a \sum_{n\ge 1} f^a_n(a,y) e^{-\frac{1}{2}a^2+ay-\frac{1}{2}y^2} \dif y\\
&\le&
\int_{0}^a e^{-\frac{1}{2}a^2+ay-\frac{1}{2}y^2} \dif y \cdot \sum_{n\ge 1} f_n \\
&=&
e^{-\beta_c} \int_{0}^a e^{-\frac{1}{2}(a-y)^2} \dif y\\
&=&
e^{-\beta_c} \int_{0}^a e^{-\frac{1}{2}y^2} \dif y\\
&\le&
a e^{-D a^2} e^{-\beta_c}.
\end{eqnarray}
(Indeed, $e^{-x}=1-x+o(x)$, so  
$\int_{0}^a e^{-\frac{1}{2}y^2} \dif y - a e^{-D a^2}= -\frac{1}{6}a^3 + D a^3 + o(a^3) $
  and thus for $D<\frac{1}{6}$ the last expression is negative whenever $a>0$ is small enough.)
Therefore the lower bound
\[
ae^{\beta_c(a)-\beta_c}\ge e^{D a^2}
\]
is achieved.
For the upper bound, note first that since $V_a$ is strictly positive \eqref{eq:estimate V} implies that
it is also (strictly) increasing on $[0,a]$. In particular, $V_a(y)\ge V_a(0)$ for all $y\in[0,a]$, and, using the
lower bound \eqref{eq: lower bound comparison fan fn}, we get
\begin{eqnarray}
e^{-\beta_c(a)} &=& \int_{0}^a \sum_{n\ge 1} f^a_n(0,y)\frac{V_{a}(y)}{V_{a}(0)}\dif y\\
&\ge&
\int_{0}^a \sum_{n\ge 1} f^a_n(a,y) \dif y\\
&\ge&
\int_{0}^a \dif y \sum_{n\ge 1} f_ne^{-C_0 a} \\
&=&
ae^{-C_0 a-\beta_c}.
\end{eqnarray}
Therefore, the upper bound
\[
ae^{\beta_c(a)-\beta_c}\le e^{C_0 a}
\]
is also achieved.

\section{Last remark on order $\frac{1}{\sqrt{N}}$ shrinking strips}\label{sec:remarks}
We actually proved that under 
$\EQset^\alpha_{\frac{c}{\sqrt{N}},\beta_c(\frac{\epsilon}{\sqrt{N}}),N}$
(or generally, under $\EQset^\alpha_{\frac{c}{\sqrt{N}},\varphi_{\frac{c}{\sqrt{N}}},N}$)
$\mathcal{A}^\alpha_{N}$ converges weakly to a random set $\mathcal{B}^\alpha$
which is absolutely continuous with respect to $\mathcal{A}^\alpha_{1/2}$. 
Moreover, for every $\epsilon>0$ we can find $0<c$ small enough so that the density $D$ is bounded by 
\[
 (1-\epsilon)e^{-\epsilon L_1} \le D \le (1+\epsilon)e^{\epsilon L_1}.
\]

\section{Acknowledgements}\label{sec:acknowledgements}
We are grateful for the financial support from the German Research Foundation through the research unit FOR 2402 – Rough paths, stochastic partial differential equations and related topics.
T.O.\ owes his gratitude to Chiranjib Mukherjee for stimulating discussions. 
He also thanks Noam Berger, Francesco Caravenna, Oren Louidor, Nicolas Perkowski, Renato Dos Santos, and Ofer Zeitouni for useful ideas.

\begin{appendices}
\section{}\label{appendix:app1}
\begin{proof}[Proof of Lemma \ref{lem:b is Hilbert-Schmidt}]
By Lemma \ref{lem:comparing excursions kernel to pinning}
 \begin{eqnarray*}
\int_{0}^{\infty} \int_{0}^{\infty} b^a_\lambda(x,y)^2 \dif x\dif y &=&
\int_{0}^{a} \int_{0}^{a} \left( \sum_{n=0}^\infty e^{-\lambda n} f^a_n(x,y)\right)
\left( \sum_{m=0}^\infty e^{-\lambda m} f^a_m(x,y)\right) \dif x\dif y \\
&=&
\int_{0}^{a} \int_{0}^{a} \sum_{n,m=0}^\infty e^{-\lambda (n+m)} f^a_n(x,y)f^a_m(x,y) \dif x\dif y \\
&=&
\sum_{n,m=0}^\infty e^{-\lambda (n+m)} \int_{0}^{a} \int_{0}^{a}  f^a_m(x,y) f^a_n(x,y) \dif x\dif y \\
&\le&
\sum_{n,m=0}^\infty e^{-\lambda (n+m)} f_n f_m\int_{0}^{a} \int_{0}^{a} \dif x\dif y \\
&\le&
 c^2 a^2\sum_{n,m=0}^\infty e^{-\lambda (n+m)} (nm)^{-3/2} \\
&=&
\left(  ca \sum_{n=0}^\infty e^{-\lambda n} n^{-3/2} \right)^2 <\infty
\end{eqnarray*}
for every $\lambda\ge 0$.
\end{proof}

\section{}\label{appendix:app2}

\begin{lem}\label{lem:enough to bound}
Let $(R_N)_{N\ge 1}$ be a sequence of non-negative random variables. Assume that there exist some $\epsilon_0>0$ and $C<\infty$ so that
$\E[e^{\epsilon_0 R_N}]\le C$ for all $N$. Then $\E[e^{\epsilon_N R_N}]\to 1$ for every sequence $\epsilon_N \to 0$.
\end{lem}
\begin{proof}
We first assume that $\epsilon_N>0$. Let $\delta>0$. It is enough to show that $\E[e^{\epsilon_N R_N}]\le 1+\delta$ for all $N$ large enough.
By Chebyshev's Inequality $\pr[R_N>r]\le Ce^{-\epsilon_0 r}$ for all $r$.
Take $r_0$ so that $Ce^{-\epsilon_0 r_0/2}<\delta/2$. It holds that
\begin{eqnarray*}
\E[e^{\epsilon_N R_N}] &=&
\E[e^{\epsilon_N R_N}\one_{R_N\le r_0}] + \E[e^{\epsilon_N R_N}\one_{R_N> r_0}] \\
&\le& e^{\epsilon_N r_0} + \E[e^{2\epsilon_N R_N}]^{1/2}\pr[R_N>r_0]^{1/2}\\
&\le& 1+ \delta/2 + C^{1/2} C^{1/2}e^{-\epsilon_0 r_0/2}\\
&\le& 1+ \delta
\end{eqnarray*}
whenever $N$ is so large so that both $e^{\epsilon_N r_0}<1+\delta/2$ and $2\epsilon_N\le \epsilon_0$ hold.
Here we used Cauchy-Schwartz in the first inequality and the fact that $\E[e^{\epsilon R_N}]$ is increasing in $\epsilon$ in the second one.
The proof for $-\epsilon_N$ is similar. Indeed,
\begin{eqnarray*}
\E[e^{-\epsilon_N R_N}] &\ge&
\E[e^{-\epsilon_N R_N}\one_{R_N\le r_0}]  \\
&\ge& e^{-\epsilon_N r_0} (1-\pr[R_N \ge r_0])\\
&\ge& (1- \frac{\delta}{2})(1- \frac{\delta}{2})\\
&\ge& 1- \delta
\end{eqnarray*}
whenever $r_0$ is chosen so that $Ce^{-\epsilon_0 r_0}\le \delta/2$
and then $N$ is so large so that $e^{-\epsilon_N r_0}\le 1-\delta/2$.
For general $\epsilon_N$'s, the lemma follows once we write them as
$\epsilon_N=\epsilon_N^+ - \epsilon_N^-$, the negative part subtracted from the positive part, and use the above on each part separately.
\end{proof}

\section{}\label{appendix:app3}
The goal of this section is to point out the connection between two definitions of the standard wetting model. 
One definition is given in the original presentation by Deuschel, Giacomin, and Zambotti \cite{deuschel2005scaling}, and Caravenna, Giacomin, and Zambotti \cite{caravenna2006sharp},
while the other is corresponding to the one e.g.\ in Giacomin \cite{giacomin2007random}, Sohier \cite{sohier2013scaling},\cite{sohier2015scaling}, 
and others, including the current paper. 
For ease of presentation we shall work in the Gaussian case $V(x)=\frac{1}{2}x^2$. 

First, we present the standard wetting model in the constraint case corresponding to \cite{deuschel2005scaling} and \cite{caravenna2006sharp} 
\begin{equation}\label{eq:StandardPinningConstrainedDef}
P_{\beta,N}^c =\frac{1}{Z_{\beta,N}^c }
\exp\left({-\frac{1}{2}\sum_{i=1}^{N}(x_{i}-x_{i-1})^2}\right)
\prod _{i=1}^N \left(\dif x_i\one_{[0,\infty)}+e^\beta \delta_0(\dif x_i)\right),
\end{equation}
for $x_0=x_N=0$, where the partition function is given by 
\begin{equation}\label{eq:pinningModGenFunc}
Z_{\beta,N}^c =
\int_{0}^{\infty}\int_{0}^{\infty}\exp\left({-\frac{1}{2}\sum_{i=1}^{N}(x_{i}-x_{i-1})^2}\right)
\prod _{i=1}^N \left(\dif x_i+e^\beta \delta_0(\dif x_i)\right).
\end{equation}
Note that $f_n=f^0_n(0,0)={(2\pi)^{-\frac{n}{2}}} Z^c_{0,n}$. 
Moreover $f_1^0(x,y) =\frac{1}{\sqrt{2\pi}}e^{-\frac{1}{2}(x-y)^2}$ for $x,y\ge0$, and one can write
\begin{eqnarray*}
f^a_n(x,y) &=&
\pr _x [S_1>0,...,S_{n-1}>0|S_n = y] \frac{1}{\dif y}\pr _x [S_n\in \dif y] \\
&=& \pr _x [S_1>0,...,S_{n-1}>0|S_n = y] \frac{1}{\sqrt{2\pi n}}e^{-\frac{1}{2n}(x-y)^2}.
\end{eqnarray*}
In the case $x=y=0$, as the increments are stationary and their density is continuous we note that $\pr _0 [S_1>0,...,S_{n-1}>0|S_n = 0]=\frac{1}{n}$
(e.g., using the following argument: for a path of size $n$ from zero to zero with distinct increments, 
the only rotation of its increments giving a path in $\mathcal{C}_N$ is the one for which the path is starting 
at its minimum, however all increments' rotations have the same probability). 
We therefore have
\begin{equation}\label{eq:f_n}
f_n = \frac{1}{\sqrt{2\pi}}n^{-3/2}.
\end{equation}
Define $\beta_c$ to be the constant so that $e^{\beta_c}\sum_{n\ge 1}f_n=1$, and set 
\[
q(n)=e^{\beta_c}f_n, \, n\ge 1,
\]
so that $q$ is a probability mass function.
Reparameterizing \eqref{eq:pinningModGenFunc} with ${\beta-\beta_c}$ and normalizing we define
\begin{equation*}
\tilde{Z}^c_{\beta,0} :=1, \, \,\tilde{Z}^c_{\beta,N} := e^{\beta-\beta_c} 
{{(2\pi)^{-\frac{N}{2}}}}Z^c_{{\beta-\beta_c},N}.
\end{equation*}
Then by summing over the first contact (i.e. the first index $1\le t\le n$ so that $S_t=0$)
we have
$$\tilde{Z}^c_{\beta,n}=e^{\beta-\beta_c} \sum_{t=1}^n q(t) \tilde{Z}^c_{\beta, n-t}.$$ 
That is,
$$\tilde{Z}^c_{\beta,n}=\sum_{k=1}^n \sum_{0=:t_0<t_1<...<t_k=n} e^{(\beta-\beta_c)k} q(t_i-t_{i-1}).$$ 
In other words, 
$$\tilde{Z}^c_{\beta,N}=\sum_{k\ge 0} e^{(\beta-\beta_c)k} q^{*k}(N)$$
where $q^{*k}(N)$ is the $k$-fold convolution of $q$ evaluated in $N$.
Therefore it is at least intuitively clear that the critical value is indeed $\beta_c$.

Analogously,
\begin{equation}\label{eq:StandardPinningFreeDef}
P_{\beta,N}^f =\frac{1}{Z_{\beta,N}^f }
\exp\left({-\frac{1}{2}\sum_{i=1}^{N}(x_{i}-x_{i-1})^2}\right)
\prod _{i=1}^N \left(\dif x_i\one_{[0,\infty)}+e^\beta \delta_0(\dif x_i)\right),
\end{equation}
with $x_0=0$, and the partition function is given by 
\begin{equation}\label{eq:StandardPinningFreeGenFcn}
Z_{\beta,N}^f =
\int_{0}^{\infty}\int_{0}^{\infty}\exp\left({-\frac{1}{2}\sum_{i=1}^{N}(x_{i}-x_{i-1})^2}\right)
\prod _{i=1}^N \left(\dif x_i+e^\beta \delta_0(\dif x_i)\right).
\end{equation}
Setting
\begin{equation*}
\tilde{Z}^f_{\beta,N} := e^{\beta-\beta_c} 
{{(2\pi)^{-\frac{N}{2}}}}Z^f_{e^{\beta-\beta_c},N},
\end{equation*}
we have 
\[
 \tilde{Z}^f_{\beta,N} =\sum_{t=1}^n \tilde{Z}^c_{\beta,t}P(N-t) 
\]
(remember the notation from \eqref{eq:def of P(n)}).
By conditioning on the contact set the original measures $P_{\beta,N}^\alpha$, $\alpha\in\{c,f\}$ are easily expressed in terms of the $\tilde{Z}^\alpha_n$, 
see (9), (13), and (17) of \cite{deuschel2005scaling}.

Now, for the strip wetting model with constant pinning, define 
\begin{equation*}
P_{a,\beta,N}^c (\dif x)=\frac{1}{Z_{a,\beta,N}^c }
\exp\left({-\frac{1}{2}\sum_{i=1}^{N}(x_{i}-x_{i-1})^2}\right)
\prod _{i=1}^N \left(\dif x_i\one_{[0,\infty)}+ e^{\beta} \one_{[0,a]}(x_i)\dif x_i\right)
\end{equation*}
for $x_0=0$,$x_n\in[0,a]$, where the partition function is given by 
\begin{equation*}
Z_{a, e^\beta,N}^c =
\int_{0}^{\infty}\int_{0}^{\infty}\exp\left({-\frac{1}{2}\sum_{i=1}^{N}(x_{i}-x_{i-1})^2}\right)
\prod _{i=1}^N \left(\dif x_i+e^\beta\one_{[0,a]}(x_i)\dif x_i\right).
\end{equation*}

We note that $P_{a,\beta,N}^c$ coincides with $\mathbb{P}_{a,\beta,N}^c$, the strip wetting model defined in 
\eqref{eq:ModelDefConstraint}.
Indeed, first note that conditioning on the contact set and the contact values the measures coincide. 
Then, we conclude using the fact that the induced measures on contact sets are proportional and hence equal.
 {
Moreover, as $a\to 0$ 
\[
\exp\left({-\frac{1}{2}\sum_{i=1}^{N}(x_{i}-x_{i-1})^2}\right)\prod _{i=1}^N (\dif x_i\one_{[0,\infty)}+ ae^{\beta(a)} \frac{1}{a}\one_{[0,a]}(x_i)\dif x_i  )
\]
convergences weakly to
\[
\exp\left({-\frac{1}{2}\sum_{i=1}^{N}(x_{i}-x_{i-1})^2}\right)\prod _{i=1}^N (\dif x_i\one_{[0,\infty)}+ e^\beta\delta_0(\dif x_i) ).
\]
whenever $\log\beta(a)+\log(a)\to \beta$ as $a\to 0$. 
Hence $\mathbb P_{a,\beta,N}^c$ is a caricature of the model corresponding to the $\delta$-pinning model 
(i.e. the standard wetting model). 
The corresponding $\tilde{Z}'s$ will be now expressed in terms of the kernel density 
$f_n^a(x,y)\frac{V_a(x)}{V_a(y)}e^{\beta_c(a)}$ 
of the corresponding Markov renewal process. The free case is analogous. 

To end, we note that the measures for a pinning function $\varphi_a$ can be define analogously. A similar argument 
then shows that if $\varphi_a$ satisfies Condition (A) and the walk's density $\rho$ is regular enough (e.g. in the Gaussian case), 
then the corresponding measure converges weakly to the critical standard wetting model. 
}
\end{appendices}

\bibliography{Wetting}
\bibliographystyle{alpha}

\end{document}